\documentclass[reqno]{amsart}
\usepackage{MnSymbol}
\usepackage{hyperref}
\usepackage{graphicx,color}
\usepackage[section]{algorithm}
\usepackage{algorithmic}
\usepackage{enumerate}
\usepackage{mathrsfs,enumerate}
\usepackage{multirow,array}
\usepackage{subfigure}
\usepackage{pstricks,pst-node,pst-tree, pst-plot, pst-eps}
\usepackage{xspace}
\usepackage{marginnote}
\usepackage{upgreek}

\chardef\atsign='100
\def\beq#1\eeq{\begin{equation}#1\end{equation}}
\def\bal#1\eal{\begin{aligned}#1\end{aligned}}
\def\blal#1\elal{\begin{flalign}#1\end{flalign}}
\def\vh{{H_h}}
\def\HO{{H^1_0}}
\def\HH{{\mathbb{H}}}
\def\Forall{{\quad\hbox{for all }}}
\def\cC{{\mathcal{C}}}
\def\cH{{\mathcal{H}}}
\def\dH{{\dot{H}}}
\def\NUMDOF{{M_h}}
\def\pih{{\pi_h}}
\def\spow{{\beta}}
\def\tpow{{\gamma}}
\def\ljh{{\lambda_{j,h}}}
\def\For{{\hbox{for }}}
\def\CC{{\mathbb{C}}}
\def\egm{{e_{\tpow,\mu}}}
\def\eg{{e_{\tpow,1}}}
\def\RR{{\mathbb{R}}}
\def\le{\leq}

\def\alal{{\alpha^*}}

\def\vh{\mathbb V_h}

\def\tt{{\mathsf T}}
\def\lan{\langle}
\def\ran{\rangle}
\def\cM{{\mathcal M}}
\def\cN{{\mathcal N}}
\def\dt{{\tau}}
\def\hc{{{\cC'}}}
\def\cE{{\mathcal E}}
\def\bE{{\overline W}}
\def\bg{{\overline g}}
\renewcommand{\Re}{\mathfrak{Re}}
\renewcommand{\Im}{\mathfrak{Im}}

\theoremstyle{plain}% default
\newtheorem{theorem}{Theorem}[section]
\newtheorem{lemma}[theorem]{Lemma}
\newtheorem{proposition}[theorem]{Proposition}
\newtheorem{corollary}[theorem]{Corollary}

\theoremstyle{theorem}
\newtheorem{definition}{Definition}[section]

\newtheorem{remark}{Remark}[section]

\newtheorem{assumption}{Assumption}[section]

\graphicspath{{figures/}}

\begin{document}

\title[APPROXIMATION OF SPACE-TIME FRACTIONAL PARABOLIC EQUATIONS]
{NUMERICAL APPROXIMATION OF SPACE-TIME FRACTIONAL PARABOLIC EQUATIONS}
\author{Andrea Bonito}
\address{Department of Mathematics, Texas A\&M University, College Station,
TX~77843-3368.}
\email{bonito\atsign math.tamu.edu}

\author{Wenyu Lei}
\address{Department of Mathematics, Texas A\&M University, College Station,
TX~77843-3368.}
\email{wenyu\atsign math.tamu.edu}

\author{Joseph E. Pasciak}
\address{Department of Mathematics, Texas A\&M University, College Station,
TX~77843-3368.}
\email{pasciak\atsign math.tamu.edu}

\date{\today}

\begin{abstract} 
In this paper, we develop a numerical scheme for the space-time fractional parabolic equation, i.e., an equation involving a fractional time derivative and a fractional spatial operator. Both the initial value problem and the non-homogeneous forcing problem (with zero initial data) are considered. The solution operator $E(t)$ for the initial value problem can be written as a Dunford-Taylor integral involving the Mittag-Leffler function $e_{\alpha,1}$ and the resolvent of the underlying (non-fractional) spatial operator over
an appropriate integration path in the complex plane. Here $\alpha$ denotes the order of the fractional time derivative. The solution for the non-homogeneous problem can be written as a convolution involving an operator $W(t)$ and the forcing function $F(t)$.   

We develop and analyze semi-discrete methods based on finite element approximation to the underlying (non-fractional) spatial operator in terms of analogous Dunford-Taylor integrals applied to the discrete operator. The space error is of optimal order up to a logarithm of $1/h$. The fully discrete method for the initial value problem is developed from the semi-discrete approximation by applying an exponentially convergent sinc quadrature technique to approximate the Dunford-Taylor integral of the discrete operator and is free of any time stepping.

To approximate the convolution appearing in the semi-discrete approximation to the non-homogeneous problem, we apply a pseudo midpoint quadrature. This involves the average of $W_h(s)$, (the semi-discrete approximation to $W(s)$) over the quadrature interval. This average can also be written as a Dunford-Taylor integral.  We first analyze the error between this quadrature and the semi-discrete approximation. To develop a fully discrete method, we then introduce sinc quadrature approximations to the
Dunford-Taylor integrals for computing the averages.
\end{abstract} 

\subjclass{65M12, 65M15, 65M60,  35S11,  65R20}

\maketitle

%%%%%%%%%%%%%%%%%%%%%%%%%%%%%%%%%%%%%%%%%%%%%%%%%%
\section{Introduction}
In this paper, we investigate  the numerical approximation
to the following time dependent problem:  
given a bounded Lipschitz polygonal domain $\Omega$, a final time $\tt>0$,
an initial value  $v\in L^2(\Omega)$ (a complex valued Sobolev space)
and a forcing  function $f\in L^\infty(0,\tt;L^2(\Omega))$,
we seek $u:[0,\tt]\times \Omega\rightarrow\RR$ satisfying
\beq\label{e:p}
\left\{ \bal
\partial_t^\tpow u+L^\spow u &= f,\qquad \text{in } (0,\tt]\times\Omega,\\
u&=0,\qquad \text{on } (0,\tt]\times \partial\Omega,\\
u&=v,\qquad \text{on } \{0\}\times\Omega. \\
\eal\right .
\eeq
Here the fractional derivative in time $\partial^\tpow_t$ with $\tpow\in (0,1)$ is defined by the left-sided
Caputo fractional derivative of order $\tpow$,
\beq\label{e:cfd}
\partial^\tpow_t u(t):=\frac{1}{\Gamma(1-\tpow)}\int_0^t \frac{1}{(t-r)^\tpow}\frac{\partial u(r)}{\partial r}\, dr.
\eeq
Note that \eqref{e:cfd} holds for smooth $u$ and extends by continuity
to a bounded operator on $H^\tpow(0,\tt)\cap C[0,\tt]$ satisfying
$$\partial^\tpow_t u={}^R\partial^\tpow_t (u-u(0))$$
where ${}^R\partial^\tpow_t$ denotes the Riemann-Liouville fractional derivative.
The differential operator $L$ appearing in \eqref{e:p} is an unbounded operator associated with a Hermitian, coercive and 
sesquilinear form $d(\cdot,\cdot)$ on $H^1_0(\Omega)\times H^1_0(\Omega)$. For $\spow\in(0,1)$,  the fractional differential
operator $L^\spow$ is defined by the following eigenfunction expansion
\beq
L^\spow v :=\sum_{j=1}^\infty  \lambda_j^ \spow (v,\psi_j)\psi_j,
\label{ldef}
\eeq
where $(\cdot,\cdot)$ denotes the $L^2(\Omega)$ inner product and
$\{\psi_j\}$ is an $L^2(\Omega)$-orthonormal 
basis of eigenfunctions of $L$ with eigenvalues $\{\lambda_j\}$.   
The above definition is valid for $v\in D(L^\spow)$, where $D(L^\spow)$ denotes the
functions $v\in L^2(\Omega)$ such that $L^\spow v\in L^2(\Omega)$.
A weak formulation of \eqref{e:p} reads: find $u\in
L^2(0,\tt;D(L^{\spow/2}))\cap C([0,\tt];L^2(\Omega))$ and $\partial_t^\tpow u \in L^2(0,\tt;D(L^{-\spow/2}))$ satisfying
\beq\label{e:wp}
\left\{ \bal
	\langle\partial_t^\tpow u, \phi\rangle+A(u,\phi)&=(f,\phi),\quad \text{for all }\phi\in D(L^{\spow/2})\text{ and for a.e. }t\in(0,\tt],\\
	u(0)&=v .
\eal\right .
\eeq
Here the bilinear form $A(u,\phi):=(L^{\spow/2}u,L^{\spow/2}\phi)$ and $\langle\cdot,\cdot\rangle$ denotes the duality pairing 
between $D(L^{-\beta/2})$ and $D(L^{\beta/2})$. 
As a consequence of \cite[Theorem 6]{NOS16},
the above problem
has a unique solution, which can be explicitly written as
\beq\label{e:sol}
u(t):=u(t,\cdot)=E(t)v+\int_0^t W({r})f(t-{r})\, d{r} .
\eeq
Here, for $w\in L^2(\Omega)$,
\beq\label{e:sol1}
E(t)w:=e_{\tpow,1}(-t^\tpow L^\spow)w=\sum_{j=1}^\infty e_{\tpow,1}(-t^\tpow \lambda_j^\spow)(w,\psi_j)\psi_j
\eeq
and
\beq\label{e:sol2}
W(t)w:=t^{\tpow-1}e_{\tpow,\tpow}(-t^\tpow L^\spow)w=\sum_{j=1}^\infty t^{\tpow-1}e_{\tpow,\tpow}(-t^\tpow \lambda_j^\spow)(w,\psi_j)\psi_j ,
\eeq
with $e_{\gamma,\mu}(z)$ denoting the Mittag-Leffler function (see the defintion \eqref{e:mlf}).
We also refer to Theorem 2.1 and 2.2 of \cite{SY11} for a detailed proof of the above formula when $\spow=1$, noting that
the argument is similar for any $\spow\in(0,1)$.

A major difficulty in approximation solutions of \eqref{e:wp} involves
time stepping in the presence of the fractional time derivative.  The L1
time stepping method was developed  in \cite{LX07} and applied for the case $\spow=1$.
Letting $\dt$ be the time step, it was shown in \cite{LX07} that the L1 scheme gives the rate of convergence 
$O(\dt^{2-\gamma})$ provided that the solution is twice continuously differentiable in
time. 
 For the homogeneous problem ($f=0$), the L1 scheme is guaranteed to yield 
first order convergence 
assuming the initial data $v$ is in $L^2(\Omega)$
(see \cite{JLZ16a}).  See also \cite{JLZ16b} and the reference therein
for other time discretization methods and error analyses.  We also 
refer to \cite{thesis} for the backward time 
stepping scheme for the case $\tpow=1$.

The numerical approximation to the solution \eqref{e:sol} has been studied recently in \cite{NOS16}. 
The main difficulty is to discretize the fractional differential operators $\partial_t^\tpow$ and $L^\spow$ simultaneously.
In \cite{NOS15}, the factional-in-space operator $L^\spow$ was approximated
as a Dirichlet-to-Neumann mapping via a Caffarelli-Silvestre extension
problem \cite{CS07} on $\Omega\times(0,\infty)$.
In \cite{NOS16}, Nochetto {\it et. al.}  analyze 
an L1  time stepping scheme for \eqref{e:wp}  in the context of the
Caffarelli-Silvestre extension problem and obtain 
a rate of convergence in time of $O(\dt^\theta)$ with $\theta\in(0,1/2)$ (see Theorem 3.11 in \cite{NOS16}).

%In this paper, we discretize the fractional-in-space operator $L^\spow$ relying on the finite element method.
%An existing finite element method in the literature to approximate $L^\spow$
%is given by \cite{NOS15,}. The fractional power of $L$ can be understood
%as a Dirichlet-to-Neumann mapping via an extension problem on $\Omega\times(0,\infty)$, which is called the Caffarelli-Silvestre extension problem \cite{CS07}. 
%%The numerical analysis of this approach to deal with the problem \eqref{e:p} is carried out in \cite{NOS16}.
%We also refer to \cite{ILIC05,ILIC06} for a finite difference approach.
%In terms of the discretizing the left-sided Caputo fractional derivative,
%a commonly used finite difference method is the so-called L1 scheme proposed by .
%Letting $\dt$ be the time step, the L1 scheme gives the rate of convergence 
%$O(\dt^{2-\gamma})$ provided that the solution is twice continuously differentiable in
%time. The first order convergence is guaranteed for the homogeneous problem $(f=0)$ with non-smooth initial data (see \cite{JLZ16a}).
%See also \cite{NOS16} for the analysis for time stepping schemes in the context of the Caffarelli-Silvertre extension problem.
%We also refer to \cite{JLZ16b} and the reference therein for more time discretization methods and error analysis.

The goal of the paper is to approximate the solution of \eqref{e:wp} directly based on the solution formula \eqref{e:sol}.
Our approximation technique and its numerical analysis 
relies on the Dunford-Taylor integral representation of the solution formula \eqref{e:sol}. 
Such a numerical method has been developed for the classical parabolic problem 
\cite{BLP17,thesis} (i.e. the case $\tpow=1$) and the stationary problem \cite{BP15}; see also \cite{BP16} when the 
differential operator $L$ is regularly accretive \cite{kato1961}. 

The outline of the remainder of the paper is as
follows. Section~\ref{s:preli} provides some notation and preliminaries
related to  \eqref{e:p}.
In Section~\ref{preli:FEM}, we review some classical results from 
the finite element discretization and provide a key result (Theorem~\ref{l:semi-sp})
instrumental to derive error estimates for semi-discrete schemes. 
%Theorem~\ref{l:semi-sp} gives an error estimate for the semi-discrete
%approximation in fractional Sobolev spaces of order $s$, with
%$s\in [0,1]$.   As expected, the rate of convergence depends on the
%smoothness of the solution which, in term, depends on the smoothness of
%the initial data and the regularity pickup associated with the spatial
%exponent $\beta$.
In Section~\ref{s:h}, we study the semi-discrete approximation
$E_h(t) v:=e_{\tpow,1}(-t^\tpow L_h^\spow)\pih v$ to  $E(t)v$.
Here $L_h$ is the Galerkin finite element approximation of $L$ in the continuous piecewise linear
finite element space $\vh$ and $\pih$ denote the $L^2$
projection onto $\vh$.  
%In Section~\ref{s:h}, we first study the semi-discrete approximation
%$E_h(t) v:=e_{\tpow,1}(-t^\tpow L_h^\spow)\pih v$ to  $E(t)v$.
We subsequently apply a sinc quadrature scheme to the Dunford-Taylor integral 
representation of the semi-discrete solution.
For the sinc approximation,
we choose the hyperbolic contour $z(y)=b(\cosh(y)+i\sinh(y))$
for $y\in \RR$, with  $b\in (0,\lambda_1/\sqrt{2})$. Here  
$\lambda_1$ denotes the smallest eigenvalue of $L$. 
Theorem~\ref{l:semi-sp} directly gives an error estimate for the semi-discrete
approximation in fractional Sobolev spaces of order $s$, with
$s\in [0,1]$.   As expected, the rate of convergence depends on the
smoothness of the solution which, in term, depends on the smoothness of
the initial data and the regularity pickup associated with the spatial
exponent $\beta$.
Theorem~\ref{l:hsincquad} proves that for  a quadrature of $2N+1$ points with
 quadrature spacing $k=cN^{-1/2}$ and  $c$ depending on $\spow$, the sinc quadrature error is bounded
by $Ct^{-\tpow}\exp(-c\sqrt{N})$, where the constant $C$ is independent of $t$ and $N$.
In Section~\ref{s:nh}, we focus on the approximation scheme for the non-homogeneous forcing problem.
The approximation in time is based on a pseudo-midpoint quadrature applied to the convolution in
\eqref{e:sol}, i.e., given a partition $\{t_j\}$ on $[0,t]$,
\beq
\int_{t_{j-1}}^{t_j} W_h({r})\pih f(t-{r})\, d{r}\approx
\bigg (\int_{t_{j-1}}^{t_j} W_h({r})\, d{r}\bigg)\ \pih f(t-{t_{j-\frac12}}) ,
\label{w-semi}
\eeq
where $W_h(t)$ is the semi-discrete approximation to $W(t)$.
Assuming that the forcing function $f$ is in $H^2(0,t;L^2)$, 
We show in Theorem~\ref{t:geo} that the error in the approximation 
\eqref{w-semi} in  time is $O(\cN^{-2})$ under a  geometric partition
refined towards $t=0$  (with $C(\tpow)\cN\log_2 \cN$ subintervals).
We then  apply an exponentially convergent sinc quadrature scheme to approximate the Dunford-Taylor integral representation of the discrete operator $\int_{t_{j-1}}^{t_j} W_h({r})\, d{r}$. 
Theorem~\ref{t:sincquad2} shows that the sinc quadrature leads to an
additional error which is $O(\log(\cN)\exp(-\sqrt{cN}))$.
Some technical proofs are given in Appendices~\ref{a:lemma1} and \ref{a:lemma2}.

Throughout this paper, $c$ and $C$ denote generic constants. 
We shall sometimes explicity indicate their dependence when appropriate.

\section{Notation and Preliminaries}\label{s:preli}
\subsection{Notation} 
Let $\Omega \subset \mathbb R^d$ be a bounded polygonal domain with Lipschitz boundary.
Denote by $L^2(\Omega)$ and $H^1(\Omega)$ (or in short $L^2$ and $H^1$) the standard Sobolev spaces of complex valued functions equipped with the norms 

$$\|u\|:=\|u\|_{L^2}:=\bigg(\int_\Omega |u|^2\, dx\bigg)^{1/2} \quad\text{and}\quad\|u\|_{H^1}:=(\|u\|^2_{L^2}+ \| | \nabla u | \|_{L^2}^2)^{1/2} .$$
The $L^2$ scalar product is denoted $(\cdot,\cdot)$:
$$
(v,w) := \int_{\Omega} v(x) \overline w(x)\, dx.
$$
We also denote by $H^1_0:=H^1_0(\Omega) \subset H^1(\Omega)$ the closed subspace of $H^1$ consisting of functions with vanishing traces. 
Thanks to the Poincar\'e inequality, we will use the semi-norm $|\cdot|_{H^1}:=\| | \nabla (\cdot) |\|$ as the norm on $H^1_0$.
The dual space of $H^1_0$ is denoted $H^{-1}:=H^{-1}(\Omega)$ and is equipped with the dual norm:
$$
\|F\|_{H^{-1}}:=\sup_{\theta\in H^1_0}\frac{\langle F,\theta\rangle}{|\theta|_{H^1}},
$$
where $\langle\cdot,\cdot\rangle$ stands for the duality pairing between $H^{-1}$ and $H^1_0$.

The norm of an operator $A: B_1 \rightarrow B_2$ between two Banach spaces $(B_1, \|.\|_{B_1})$ and $(B_2, \|.\|_{B_2})$ is given by
$$
\| A \|_{B_1 \to B_2} = \sup_{v \in B_1, \ v \not = 0} \frac{\|A v\|_{B_2}}{\| v \|_{B_1}}
$$
and in short $\| A\|$ when $B_1 = B_2 = L_2$.

\subsection{The Unbounded Operator $L$}\label{ss:unbounded} Let us assume that  $d(\cdot,\cdot)$ is a Hermitian, coercive and 
sesquilinear form on $\HO\times\HO$.
We denote by $c_0$ and $c_1$ the two positive constants such that
$$
 c_0 | v |^2_{H^1}\le d(v,v); \qquad  |d(v,w)| \leq c_1 |
v|_{H^1} | w |_{H^1},\Forall v,w\in H^1_0.
$$
Furthermore, we let $T:H^{-1} \rightarrow
H^1_0$ be the solution operator, i.e. for $f \in H^{-1}$,  $Tf:=w \in H^1_0$, where $w$ is the unique solution (thanks to Lax-Milgram lemma) of 
\beq
d(w,\theta)= \langle f,\theta \rangle ,\Forall \theta\in H^1_0 .
\label{laxm}
\eeq
%Here $(\cdot,\cdot)$ is the $L^2(\Omega)$ inner product.
Following Section 2 of \cite{kato1961}, see also Section~2.3 in \cite{BP16}, we denote $L$ to be the inverse of $T|_{L^2}$ and define $D(L):=\text{Range}(T|_{L^2})$.

\subsection{The Dotted Spaces}\label{ss:dotted}
The operator $T$ is compact and symmetric 
on $L^2$. 
Fredholm theory guarantees the existence of an $L^2$-orthonormal
basis of eigenfunctions $\{\psi_j\}_{j=1}^\infty$ with non-increasing 
real eigenvalues $\mu_1\geq\mu_2 \geq \mu_3 \geq ... >0$. 
For every positive integer $j$, $\psi_j$ is 
also an eigenfunction of $L$ with corresponding eigenvalue $\lambda_j=1/\mu_j$.
The decay of the coefficients $(v,\psi)$ in the representation
$$
v = \sum_{j=1}^\infty (v,\psi_j) \psi_j 
$$
characterizes the dotted spaces $\dH^s$. 
Indeed, for $s\geq 0$, we set
$$
\dH^s:= \left\lbrace v\in L^2\ \text{s.t.} \  \sum_{j=1}^\infty \lambda_j^s |(v,\psi_j)|^2<\infty
\right\rbrace.
$$
On $\dH^s$, we consider the natural norm
 $$
 \|v\|_{\dH^s} := \bigg(\sum_{j=1}^\infty \lambda_j^s |(v,\psi_j)|^2\bigg)^{1/2} .
 $$ 
We also denote by $\dH^{-s}$ the dual space of 
$\dH^s$ for $s\in [0,1]$. It is known that (see for instance \cite{BP16})
%i.e. if $\lan F,\psi_j\ran$ denotes the action of $F \in \dH^{-s}(\Omega)$ applied to $\psi_j$ then
$$
\dH^{-s}= \left\{F\in H^{-1} \ \text{s.t} \ \|F\|_{\dH^{-s}}:= \bigg(\sum_{j=1}^\infty
\lambda_j^{-s}|\lan F,\psi_j\ran|^{2}\bigg)^{1/2}<\infty\right \}.
$$
Note that, we identify $L^2$ functions with $H^{-1}$ functionals by
$\langle F,\cdot\rangle :=(f,\cdot)\in H^{-1}$ for $f \in L^2$.
\subsection{Fractional Powers of Elliptic Operators}
Let $L$ be defined from a Hermitian, coercive and 
sesquilinear form on $\HO\times\HO$ as described in Section~\ref{ss:unbounded}.
For $\spow\in (0,1)$, the fractional power of $L$ is given by
\beq\label{spe}
L^\spow v := \sum_{j=1}^\infty \lambda_j^\spow (v,\psi_j)\psi_j, \qquad \forall v \in D(L^\spow):=\dH^{2\spow}.
\eeq
In addition, we define the associated sesquilinear form $A: \dH^\spow \times \dH^\spow \rightarrow \mathbb C$ by
\beq\label{Asf}
 A(v,w):= (L^{\spow/2} v,L^{\spow/2} w) =\sum_{j=1}^\infty \lambda_j^{\spow} (v,\psi_j)
  \overline{(w,\psi_j)},
\eeq
which satisfies $A(v,v)=\|v\|_{\dH^{\spow}}^2$.

\subsection{Intermediate Spaces and the Regularity Assumption}
As we saw above, the dotted spaces relies on the eigenfunction decomposition of a compact operator. 
These are natural spaces to consider fractional powers of operators but are less adequate to describe standard smoothness properties.
The latter are better characterized by the intermediate spaces $\HH^s$ defined for $s\in[-1,2]$ by real interpolation
\begin{equation}\label{e:interpolation_spaces}
\HH^s:=\left\{\bal &[\HO, H^1_0\cap H^2]_{s-1,2},&  1\le s\le 2,\\
				  &[L^2,H^1_0]_{s,2},&  0 \le s \le 1,\\
				  &[H^{-1},L^2]_{s+1,2},& -1\le s \le 0.\eal \right.
				  \end{equation}
%Note that we identify a $L^2$ function $f$ in $H^{-1}$ by letting $F=(f,\cdot)\in H^{-1}$.

In order to link the two set of functional spaces introduced above, we assume
the following elliptic regularity condition:

\begin{assumption}[Elliptic Regularity]\label{regularity}
There exists $\alpha\in(0,1]$ so that
\begin{enumerate}[(a)]
 \item  $T$ is a bounded map of $\HH^{-1+\alpha}$ into $\HH^{1+\alpha}$;
 \item  $L$ is a bounded operator from $\HH^{1+\alpha}$ to 
$\HH^{-1+\alpha}$.
\end{enumerate}
\end{assumption}

Under the above assumption we have the following equivalence property:
\begin{proposition}[Equivalence, Proposition 4.1 in \cite{BP15}]\label{p:equiv}
Suppose that Assumption~\ref{regularity} holds for $\alpha \in (0,1]$. 
Then the spaces $\HH^s$ and $\dot{H}^s$ coincide for $s\in[-1,1+\alpha]$ with equivalent norms. 
\end{proposition}

Notice that Assumption~\ref{regularity} is quite standard and holds for a large class of sesquilinear forms $d(\cdot,\cdot)$.
An important example is the diffusion process given by
$$
d(u,v) = \int_\Omega a(x) \nabla u \cdot \nabla v\, dx
$$
defined on $H^1_0 \times H^1_0$, 
where $a\in L^\infty(\Omega)$ satisfies 
$$
0<c_0 \leq a(x) \leq c_1 \qquad \textrm{for a.e.}\quad  x\in \Omega.
$$
The $\alpha$ in Assumption ~\ref{regularity} is related to the domain
$\Omega$ and the smoothness of the coefficients.  For example, if
$\Omega$ is convex and $a$ is smooth, Assumption ~\ref{regularity} holds
for any $\alpha$ in $(0,1]$.  In contrast, for the two dimensional L-shaped domain and
smooth $a$, Assumption ~\ref{regularity} only holds for
$\alpha\in(0,2/3)$.

\subsection{The Mittag-Leffler Function}
The Mittag-Leffler functions are instrumental to represent the solution of fractional time evolution, see \eqref{e:sol1} and \eqref{e:sol2}.
We briefly introduce them together with their properties used in our argumentation. We refer to Section 1.8 in \cite{KST06} for more details.

For $\tpow>0$ and $\mu\in\RR$,
the two-parameter Mittag-Leffler function $\egm(z)$ is defined by
\beq\label{e:mlf}
	\egm(z):=\sum_{k=0}^\infty\frac{z^k}{\Gamma(k\tpow+\mu)},\qquad z\in\CC .
\eeq
These functions are  entire functions (analytic in $\CC$).  We note that
\cite[equation (3.1.42)]{KST06} (see also \cite{MR2053894}) implies that $u(t)=\eg(-\lambda t^\tpow)$ for
$t,\lambda>0$ satisfies
$$
\partial^\tpow_t u +\lambda u=0,
$$
i.e., is a solution of the scalar homogeneous version of the first
equation of \eqref{e:p}.  For this reason, the function  $\eg(-\lambda
t^\tpow)$ will play a major role in our analysis.   We also note that 
\begin{equation}\label{e:fd2}
\partial_t e_{\tpow,1}(-t^\tpow\lambda^\spow)=\lambda^\spow t^{\tpow-1}e_{\tpow,\tpow}(-t^\tpow\lambda^\spow)
\end{equation}
and
\begin{equation}\label{e:fd2_2}
\partial_t e_{\tpow,\tpow}(-t^\tpow\lambda^\spow)=\lambda^\spow t^{\tpow-1}\left((\gamma-1)e_{\tpow,2\tpow}(-t^\tpow\lambda^\spow)
-e_{\tpow,2\tpow-1}(-t^\tpow\lambda^\spow)\right).
\end{equation}
Recall that $\partial^\gamma_t$ always denotes the left-sided Caputo fractional derivative \eqref{e:cfd}.

Another critical property for our study is their decay when $|z|\to \infty$ in a positive sector: For $0<\tpow<1$, $\mu\in\RR$ and $\frac{\tpow\pi}{2}<\zeta< \tpow\pi$,
there is a constant $C$ only depending on $\tpow,\mu,\zeta$ so that
\beq\label{ml-bound-scalar}
	|\egm(z)|\le\frac{C}{1+|z|},\quad\hbox{for }\zeta\le |\arg(z)|\le\pi .
\eeq

\subsection{Solution via superposition}\label{s:weak_hom}

The solution $u$ of \eqref{e:wp} is the superposition of two solutions: the homogeneous solution $f=0$ and the non-homogeneous solution $v=0$,
\begin{equation}\label{e:full_sol}
u(t) = E(t)v+\int_{0}^t W(s) f(t-s)~ds,
\end{equation}
 where $E(t)$ is defined by \eqref{e:sol1} and $W(t)$ by \eqref{e:sol2}.
Following \cite{SY11}, we have that $u \in C^0([0,T];L^2)$ and in particular
$ u(0)=v$.

We discuss the approximation of each term in the decomposition separately.
For the homogeneous problem ($f=0$), we use the Dunford-Taylor integral representation of $u(t)= E(t)v$,
\beq\label{DFint}
u(t) = \frac 1 {2\pi i} \int_\cC e_{\tpow,1}({-t^\tpow z^\spow}) R_z(L)v\, dz.
\eeq
Here $R_z(L):=(zI-L)^{-1}$ and $z^{\spow}:= e^{{\spow} \ln z}$ with the logarithm defined with branch
cut along the negative real axis.
Given $r_0\in (0,\lambda_1)$, the contour $\cC$ 
consists of three segments (see Figure~\ref{f:anal}):
\begin{equation}\label{e:contour}
\begin{aligned}
\cC_1&:=\left\lbrace z(r):=re^{-i\pi/4}  \text{ with  }r  \text{ real going from }+\infty \text{ to } r_0\right\rbrace \text{ followed by } \\
\cC_2&:=\left\lbrace z(\theta):=r_0 e^{i\theta} \text{ with }\theta \text{ going from }-\pi/4 \text{ to } \pi/4\right\rbrace \text{ followed by } \\
\cC_3&:=\left\lbrace z(r):=re^{i\pi/4} \text{ with } r \text{ real going from } r_0 \text{ to }+\infty \right\rbrace.
\end{aligned}
\end{equation}
\begin{figure}[H]
\centering
\includegraphics[scale=.4]{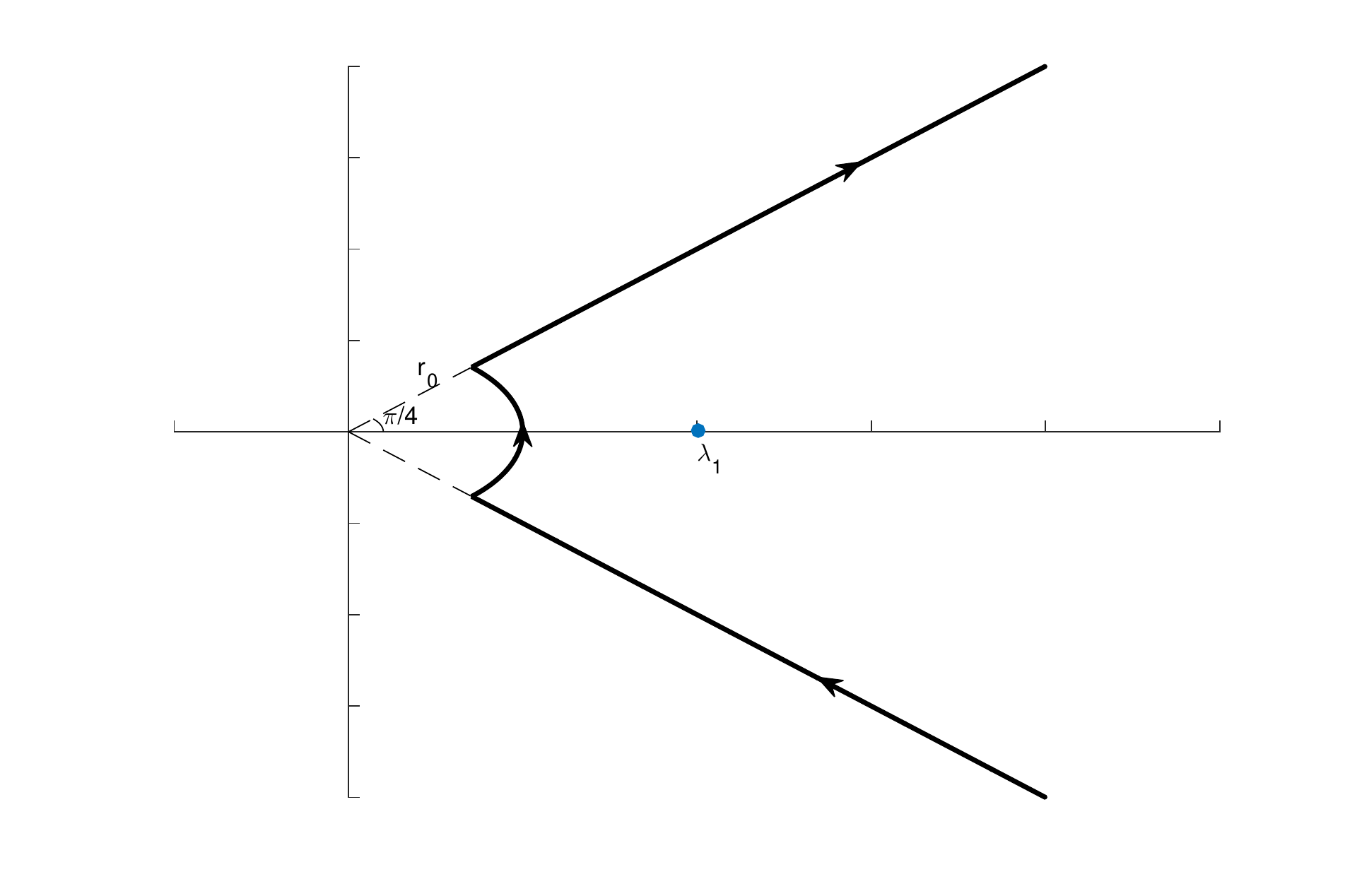}
\caption{The contour $\cC$ given by \eqref{e:contour}.}
\label{f:anal}
\end{figure}

We use an analogous representation for $W(s)$, namely,
\beq
\label{Wsint}
W(s)v = \frac {s^{\gamma-1}} {2\pi i} \int_\cC e_{\tpow,\tpow}({-s^\tpow z^\spow}) R_z(L)v\, dz.
\eeq
The justification of \eqref{DFint} and \eqref{Wsint} are a consequence
of \eqref{ml-bound-scalar} and standard Dunford-Taylor integral
techniques, see, \cite{yoshida,MR1192782} for additional details.

\section{Finite Element Approximations}\label{preli:FEM}

\subsection{Subdivisions and Finite Element Spaces } 
Let $\{\mathcal{T}_h\}_{h>0}$ be a sequence of globally shape regular and quasi-uniform 
conforming subdivisions of $\Omega$ made of simplexes, i.e. there are positive
constants  $\rho$ and $c$ independent
of $h$ 
such that if for $\uptau \in \mathcal T_h$, $h_\uptau$ denotes the diameter of $\uptau$ and $r_\uptau$
denotes the radius of the largest ball which can be inscribed in $\uptau$,
then
\blal\label{shape-regular}
&\hbox{(shape regular)}\qquad \max_{\uptau \in \mathcal T_h} \frac{h_\uptau}{r_\uptau} \le c,\quad \hbox{and}\\
\label{ineq:quasi}
 &\hbox{(quasi-uniform)}\qquad  \max_{\uptau\in\mathcal{T}_h}h_\uptau \le \rho\min_{\uptau\in\mathcal{T}_h}
  h_\uptau.
\elal
Fix $h>0$ and 
denote by $\vh\subset \HO$ the space of continuous piecewise linear finite element functions
with respect to $\mathcal{T}_h$ and by $\NUMDOF$ the dimension of $\vh$. 

The $L^2$ projection onto $\vh$ is denoted by $\pih: L^2 \rightarrow \vh$ and satisfies
$$
(\pih f,\phi_h)=(f,\phi_h),\qquad\Forall \phi_h\in\vh .
$$
For $s\in [0,1]$ and $\sigma>0$ satisfying $s+\sigma\le 2$, Lemma~5.1 in \cite{BP16} guarantees the existence of a constant $c(s,\sigma)$ independent of $h$ such that 
\begin{equation}\label{pih-approx}
\|(I-\pi_h)f\|_{\HH^s} \le c(s,\sigma) h^{\sigma}\|f\|_{\HH^{s+\sigma}} .
\end{equation}
In addition, for any $s\in [0,1]$, there exists a constant $c$ such that
\begin{equation}\label{pih-bound}
\|\pi_h f\|_{\HH^{s}}\le c \|f\|_{\HH^{s}}.
\end{equation}
The case $s=0$ follows from the definition of the $L^2$-projection, the case $s=1$ is treated in \cite{BY14,BX91} and the general case follows  by interpolation. 

\subsection{Discrete Operators}\label{ss:discrete_op}
The finite element analogues of the operators $T$ and $L$ given in Section~\ref{ss:unbounded} are defined as follows:
For $F\in H^{-1}$, $T_h:H^{-1}\rightarrow\vh$ is defined by 
$$d(T_h F,\phi_h)=\langle F,\phi_h \rangle,  \Forall \phi_h\in \vh$$
and $L_h:\vh \rightarrow \vh$ is given by
$$
(L_h v_h,\phi_h)=d(v_h,\phi_h), \Forall \phi_h\in \vh.
$$
so that $T_h|_{\vh}=L_h^{-1}$. 

 We now recall the following finite element error estimates.
\begin{proposition}[Lemma 6.1 in \cite{BP16}]\label{prop:T_error}
  Let Assumption~\ref{regularity}~(a) holds for some $\alpha \in (0,1]$. Let $s\in [0,\frac 1 2]$ and set $\alpha^*:=\frac 1 2 (\alpha+\min(\alpha,1-2s))$. There is a constant $C$ independent of $h$ such that   \beq\label{duality}
  \|T-T_h\|_{\dot{H}^{\alpha-1}\to \dot{H}^{2s}}\le C h^{2\alpha^*}.
  \eeq
\end{proposition}

Similar to the operator $T$, $T_h|_{\vh}$ has positive 
eigenvalues $\{\mu_{j,h}\}_{j=1}^\NUMDOF$  with corresponding 
$L^2$-orthonormal eigenfunctions $\{\psi_{j,h}\}_{j=1}^\NUMDOF$. 
The eigenvalues of $L_h$ are denoted
as $\lambda_{j,h}:=\mu_{j,h}^{-1}$
 for $j=1,2,\ldots,\NUMDOF$.
Then the discrete fractional operator $L_h^\spow\ :\ \vh \rightarrow \vh$ is then given by
$$
L_h^\spow v_h:=\sum_{j=1}^\NUMDOF \lambda_{j,h}^\spow (v_h,\psi_{j,h})\psi_{j,h}.$$
Its associated sesquilinear form reads
\begin{equation}\label{e:Ah}
A_h(v_h,w_h):=(L_h^{\beta/2} v_h, L_h^{\beta/2} w_h) = \sum_{j=1}^\NUMDOF
 \lambda_{j,h}^\spow (v_h,\psi_{j,h})\overline{(w_h,\psi_{j,h})},
\end{equation}
for all $v_h, w_h\in \vh$.

For any $s\in [0,1]$, the dotted spaces described in Section~\ref{ss:dotted} also have discrete counterparts $\dot{H}_h^s$,
which are characterized by their norms 
\begin{equation}\label{e:dotted_discrete_norm}
\|v_h\|_{\dot{H}_h^s}:=\Bigg(\sum_{j=1}^\NUMDOF
\lambda_{j,h}^s|(v_h,\psi_{j,h})|^2\Bigg)^{1/2},\qquad \text{for } v_h\in \vh .
\end{equation}
On $\vh$, the two dotted norms are equivalent: For $s\in[0,1]$, there exists a constant $c$ independent of $h$ such that for all  $v_h\in \vh$,
\begin{equation}\label{ineq:H_h_H}
  \frac{1}{c}\|v_h\|_{\dot{H}_h^s}\le\|v_h\|_{\dot{H}^s}\le c\|v_h\|_{\dot{H}_h^s},
\end{equation}
(see Appendix A.2 in \cite{BZ00}).
From the property $\max_j \ljh \le ch^{-2}$ (cf. \cite[equation (2.8)]{BZ00}), we also deduce an inverse inequality in discrete dotted spaces:
For $s,\sigma\ge 0$, we have
\beq\label{inverse}
	\|v_h\|_{\dH_h^{s+\sigma}}\le c h^{-\sigma}\|v_h\|_{\dH_h^s},\qquad\For v_h\in \vh .
\eeq

\subsection{The Semi-discrete Scheme in Space}
We propose a Galerkin finite element method for the space discretization of \eqref{e:sol}.
This is to find $u_h(t)\in \vh$
satisfying 
\begin{equation}\label{e:weakfe}
\left\lbrace
\begin{aligned}
  (\partial^\tpow_t u_{h}(t),\phi_h)+A_h(u_h(t),\phi_h)&=(f,\phi_h), \qquad \text{for }
  t\in (0,\tt],\ \hbox{and } \phi_h\in \vh\text{, and}\\
  u_h(0)&=\pi_h v,
\end{aligned}
\right.
\end{equation}
where the bilinear form $A_h(\cdot,\cdot)$ is defined by \eqref{e:Ah} and $\pi_h$ is the $L^2$-projection onto $\vh$.
Similarly to the continuous case (see discussion in Section~\ref{s:weak_hom}), the solution of the above discrete problem is given by
\beq\label{e:dsol1}
	u_h(t)=\underbrace{e_{\tpow,1}(-t^\tpow L_h^\spow)}_{=:E_h(t)}\pih v +\int_0^t \underbrace{s^{\gamma-1} e_{\gamma,\gamma}(-s^\gamma L_h^\spow)}_{=:W_h(s)} \pih f(t-s)\, ds 
\eeq
where
\beq\label{e:fesol-dt}
	e_{\gamma,\mu}(-t^\gamma L_h^\beta) = \frac{1}{2\pi i}\int_{\cC} e_{\tpow,\mu}(-t^\tpow z^\spow) R_z(L_h)\, dz
\eeq
and $\cC$ is as in \eqref{e:contour}.

\subsection{A semi-discrete estimate}

The purpose of this section (Theorem~\ref{l:semi-sp}) is to obtain
estimates for 
\begin{equation}\label{e:to_prove_semi}
\| e_{\gamma,\mu}(-t^\gamma L^\beta)v - e_{\gamma,\mu}(-t^\gamma L_h^\beta)\pih v\|_{\dH^{2s}},
\end{equation}
which, in view of representations \eqref{e:full_sol} and \eqref{e:dsol1}, will be instrumental to derive error estimates for the space discretization.

The following lemma assesses the discrepancy between the resolvant $R_z(L)=(z-L)^{-1}$ and its finite element approximation.
Its somewhat technical proof is postponed to Appendix~\ref{a:lemma1}.

\begin{lemma}[Space Discretization of Resolvant]\label{l:residue}
Assume that Assumption~\ref{regularity} holds for some $\alpha\in(0,1]$. Let $s \in  [0,\frac 12]$ and $\delta \in [0, (1+\alpha)/2]$.
Then, there exists a positive constant $C$ independent of $h$ such that for all $\tilde \alpha$ with $2\tilde \alpha \in (0,\alpha+\min(\alpha,1-2s)]$,  $z\in\cC$ and 
$v\in\dH^{2\delta}$
\beq\label{rz-bound}
\|(\pi_h R_z(L)-R_z(L_h)\pi_h)v\|_{\dH^{2s}}\le C |z|^{-1+\tilde \alpha+s-\delta} h^{2\tilde \alpha} \|v\|_{\dH^{2\delta}}.
\eeq
\end{lemma}

We are now in a position to prove the error estimate for the semi-discrete approximation in space.
Before doing so,
for $s \in [0,1/2]$ and $0<\epsilon \ll 1$, we set
\begin{equation}\label{e:astar}
\alpha^* :=  \alpha/ 2+ \min\{\alpha/2,1/2-s,\spow+\delta-s-\alpha/2-\epsilon/2\}.
\end{equation}
We assume that
\beq\label{delta}
\delta  \geq \max\{0,s-\beta+\epsilon/2\}.
\eeq
The assumption \eqref{delta} is sufficient to guarantee  that  the
solution  
$ e_{\gamma,\mu} (-t^\gamma L^\beta) $ is in $\dH^{2s+\epsilon}$ and we
have the following theorem.

\begin{theorem}[Space Discretization of $e_{\gamma,\mu}(-t^\tpow L^\spow)$]\label{l:semi-sp}
Let $0<\gamma<1$, $s\in [0,1/2]$, $\mu \in \mathbb R$ and   $\alpha^*$ be as in \eqref{e:astar}.
Assume that Assumption~\ref{regularity} holds for  $\alpha\in (0,1]$,
and that $\delta $ satisfies \eqref{delta}.
Then there exists a constant $C$ such that 
 $$
 \|e_{\gamma,\mu}(-t^\tpow L^\beta)- e_{\gamma,\mu}(-t^\tpow L_h^\beta)\pi_h \|_{\dH^{2\delta} \to \dH^{2s}}\le  D(t)h^{2\alal},
 $$
where
    \begin{equation}\label{e:cd}
     D(t):=\left \{\bal C: & \qquad\hbox{if } \delta>\alal+s, \\
     C\max(1,\ln(t^{-\gamma})):&\qquad\hbox{if } \delta=\alal+s,\\
     Ct^{-{\tpow(\alal+s-\delta)}/{\spow}}:& \qquad\hbox{if } \delta<\alal+s.\\ 
    \eal \right. 
\end{equation}
\end{theorem}
\begin{proof}
Without loss of generality we assume that $2\delta \leq 1+\alal$ as the case $2\delta >  1+\alal$ follows from the continuous embedding
$$
\dH^{2\delta}\subset \dH^{1+\alal}.
$$
Also, we use the notation $E^{\gamma,\mu}(t) := e_{\gamma,\mu}(-t^\tpow L^\beta)$, $E^{\gamma,\mu}_h(t):=e_{\gamma,\mu}(-t^\tpow L_h^\beta)$ and decompose the error in two terms:
\begin{equation}\label{e:decomp}
\begin{split}
\|(E^{\gamma,\mu}(t)-E^{\gamma,\mu}_h(t)\pi_h)v\|_{\dH^{2s}}
& \le \|(I-\pi_h)E^{\gamma,\mu}(t)v\|_{\dH^{2s}}\\
& + 
\|\pi_h (E^{\gamma,\mu}(t)-E^{\gamma,\mu}_h(t)\pi_h)v\|_{\dH^{2s}}.
\end{split}
\end{equation}

$\boxed{1}$ For the first term on the right hand side above, we note that the assumptions on the parameters imply that $\alal+s\le (\alpha+1)/2\le 1$
and so the approximation property \eqref{pih-approx} of $\pi_h$ yields
\begin{equation}\label{e:thm_proj}
\|(I-\pi_h)E^{\gamma,\mu}(t)v\|_{\dH^{2s}}\leq Ch^{2\alal}\|E^{\gamma,\mu}(t) v\|_{\dot{H}^{2(\alal+s)}} .
\end{equation}
We estimate $\|E^{\gamma,\mu}(t)v\|_{\dot{H}^{2(\alal+s)}}$ by expanding $v$ in Fourier series with respect to the eigenfunctions of $L$ (see Section~\ref{ss:dotted})
and denote by  $c_j:=(v,{\psi}_j)$ the Fourier coefficient of $v$ so that
$$
E^{\gamma,\mu}(t)v = \sum_{j=1}^{\infty}e_{\gamma,\mu}(-t^\tpow\lambda_j^\spow) c_j \psi_j.
$$

 Two cases need to be considered:

Case 1: $\delta\ge\alal+s$. Here, the regularity of the initial condition is large enough to directly use the bound $|\egm(-t^\gamma \lambda_j^\beta)|\le C$  deduced from  \eqref{ml-bound-scalar} to get
$$\bal
\|E^{\gamma,\mu}(t)v\|^2_{\dH^{2\alal+2s}} &=\sum_{j=1}^{\infty}\lambda_j^{2\alal+2s}|e_{\gamma,\mu}(-t^\tpow\lambda_j^\spow)|^2 |c_j|^2\\
&\le C\lambda_1^{2(\alal+s-\delta)}\sum_{j=1}^{\infty}\lambda_j^{2\delta}|c_j|^2=C\lambda_1^{2(\alal+s-\delta)}
\|v\|_{\dot{H}^{2\delta}}^2.
\eal$$

Case 2: $\delta < \alal+s$. In this case, we need to rely on the parabolic regularity for $t>0$. 
We apply
\eqref{ml-bound-scalar} again and obtain
$$\bal
\|E^{\gamma,\mu}(t)v\|_{\dH^{2\alal+2s}}^2 
&=t^{-2\tpow(\alal+s-\delta)/{\spow}}\sum_{j=1}^\infty
\lambda_j^{2\delta}
\left| (t^\tpow\lambda_j^\spow)^{(\alal+s-\delta)/\spow}e_{\gamma,\mu}(-t^\tpow \lambda_j^\spow)\right|^2 
|c_j|^2\\
&\leq C t^{-2\tpow(\alal+s-\delta)/{\spow}}\sum_{j=1}^\infty
\lambda_j^{2\delta}
\left|\frac{(t^\tpow\lambda_j^\spow)^{(\alal+s-\delta)/\spow}}{1+t^\tpow\lambda_j^\spow}\right|^2 
|c_j|^2.
\eal
$$
Noting that $0<\alal+s-\delta<\spow$, a Young's inequality implies
$$
\left|\frac{(t^\tpow\lambda_j^\spow)^{(\alal+s-\delta)/\spow}}{1+t^\tpow\lambda_j^\spow}\right|\le 1.
$$
Whence, 
$$
\|E^{\gamma,\mu}(t)v\|_{\dH^{2\alal+2s}}^2 
\leq
Ct^{-2\tpow(\alal+s-\delta)/{\spow}}
\|v\|^2_{\dot{H}^{2\delta}}.
$$
Returning to \eqref{e:thm_proj} after gathering the estimates obtained for the two different cases, we obtain
\begin{equation}\label{e:thm_proj2}
\|(I-\pi_h)E^{\gamma,\mu}(t)v\|_{\dH^{2s}}\leq D(t) h^{2\alal}\|v\|_{\dot{H}^{2\delta}}.
\end{equation}

$\boxed{2}$
We return to \eqref{e:decomp} and estimate now  $\|\pi_h (E(t)-E_h(t)\pi_h)v\|_{\dH^{2s}}$. This time we use the integral representations and the resolvant approximation (Lemma~\ref{l:residue}) to get
$$
\bal
\|\pi_h(E^{\gamma,\mu}(t)-E^{\gamma,\mu}_h(t)\pi_h)v\|_{\dH^{2s}}&\le C\int_{\cC}|e_{\gamma,\mu}(-t^\tpow z^\spow)| |(\pi_h R_z(L)-R_z(L_h)\pi_h)v\|_{\dH^{2s}} \, d|z|\\
&\le Ch^{2\alal} \| v \|_{\dH^{2\delta}}\int_{\cC}|e_{\gamma,\mu}(-t^\tpow z^\spow)| |z|^{-1+\alal+s-\delta}  \, d|z|.
\eal
$$
Furthermore, the decay estimate \eqref{ml-bound-scalar} of the Mittag-Leffler function evaluated at $-t^\tpow z^\spow$ for $z \in \cC$ yields
\begin{equation}\label{e:thm2}
\|\pi_h(E^{\gamma,\mu}(t)-E^{\gamma,\mu}_h(t)\pi_h)v\|_{\dH^{2s}} \le Ch^{2\alal} \| v \|_{\dH^{2\delta}} \int_{\cC}\frac{ |z|^{-1+\alal+s-\delta}}{1+t^\tpow |z|^\spow} \, d|z|.
\end{equation}

$\boxed{3}$ To prove 
\begin{equation}\label{e:thm3}
\|\pi_h(E^{\gamma,\mu}(t)-E^{\gamma,\mu}_h(t)\pi_h)v\|_{\dH^{2s}} \le D(t) h^{2\alal} \| v \|_{\dH^{2\delta}},
\end{equation}
it remains to show that
\begin{equation}\label{e:to_show_D}
\int_{\cC}\frac{ |z|^{-1+\alal+s-\delta}}{1+t^\tpow |z|^\spow} \, d|z| \leq D(t)
\end{equation}
This is done separately on each part of the contour $\cC$, see \eqref{e:contour}. 
On $\cC_2$, $|z|= r_0$ so that we directly have 
$$
\int_{\cC_2}\frac{|z|^{-1+\alal+s-\delta}}{1+ t^\tpow |z|^\spow}  \, d|z|
\le \int_{\cC_2}|z|^{-1+\alal+s-\delta} \, d|z|\le C.
$$
On $\cC_1\cup\cC_3$, we use the parametrization $z(r) = re^{{\pm}i\pi/4}$ to write
$$
\int_{\cC_1\cup\cC_3}\frac{|z|^{-1+\alal+s-\delta}}{1+ t^\tpow |z|^\spow}\, d|z|
=2\int_{r_0}^\infty \frac{r^{-1+\alal+s-\delta}}{1+ t^\tpow r^\spow}\, dr .$$
When $\delta>\alal+s$, we have enough decay to directly obtain
$$
\int_{\cC_1\cup\cC_3}\frac{|z|^{-1+\alal+s-\delta}}{1+ t^\tpow |z|^\spow}\, d|z| \le 2\int_{r_0}^\infty r^{-1+\alal+s-\delta}\, dr\le C .
$$
When $\delta \leq \alal + s$, we perform the change of variable
$y:=t^\tpow |z|^\spow$ and obtain
\beq\label{I-bound}
\int_{\cC_1\cup\cC_3}\frac{|z|^{-1+\alal+s-\delta}}{1+ t^\tpow |z|^\spow}\, d|z| = \frac 2 \beta t^{-\frac{\tpow(\alal+s-\delta)}{\spow}}\int_{t^\tpow r_0^\spow}^\infty
\frac{y^{{(\alal+s-\delta)}/{\spow}-1}}{1+y}\, dy. 
\eeq
Thus,
\begin{equation*}
\begin{split}
&\int_{\cC_1\cup\cC_3}\frac{|z|^{-1+\alal+s-\delta}}{1+ t^\tpow |z|^\spow}\, d|z| \\
& \qquad \le  \frac 2 \beta
t^{-{\tpow(\alal+s-\delta)}/{\spow}}\left(\int_
{t^\tpow r_0^\spow}^1
y^{\frac{\alal+s-\delta}{\spow}-1}\, dy+\int_1^\infty y^{\frac{\alal+s-\delta}{\spow}-2}\, dy\right) \\
& \qquad \le C \left\lbrace
\begin{array}{ll}
t^{-{\tpow(\alal+s-\delta)}/{\spow}}, &\quad \textrm{when }\delta < \alal +s, \\
\max(1,\ln(t^{-\gamma})), & \quad \textrm{when } \delta = \alal+s.
\end{array}
\right.
\end{split}
\end{equation*}
$\boxed{4}$ Gathering the estimates for each part of the contour yields \eqref{e:to_show_D} and thus \eqref{e:thm3}, which, combined with \eqref{e:thm_proj2}, yields the desired result.
\end{proof}

%%%%%%%%%%%%%%%%%%%%%%%%%%%%%%%%%%%%%%%%%%%%%%%%%%
\section{Approximation of the Homogeneous Problem}\label{s:h}

This section presents and analyzes the proposed approximation algorithm in the case $f=0$. 
We note that the bound for the finite element approximation  for the
space discretization 
error is contained in Theorem~\ref{l:semi-sp}.  In this section, we
define a sinc quadrature approximation to $E_h(t)$ and analyze the
resulting quadrature error.

\subsection{The Sinc Quadrature Approximation}\label{sinc}

We discuss the approximation of the contour integral 
 in  
$$
u_h(t) = e_{\gamma,1}(-t^\gamma L_h^\beta) \pi_h v = \frac 1 {2\pi i} \int_{\cC} e_{\gamma,1}(-t^\gamma z^\beta) R_z(L_h) \pi_h v \,dz.
$$
The first step involves replacing the  contour $\cC$  by one  more
suitable for application of the  sinc quadrature technique.
For $y\in \CC$, we set 
\beq 
z(y) =b(\cosh{ y}+i\sinh{ y})
\label{hc}
\eeq
and, for $0<b<\lambda_1/\sqrt{2}$,  
consider the hyperbolic contour $\hc:=\{z(y)\ : \ y\in \RR\}$.
Using this contour, we have
$$
e_{\gamma,1}(-t^\gamma L_h^\beta) g_h = \frac{1}{2\pi i} \int_{-\infty}^\infty e_{\tpow,1}({-t^\tpow z(y)^\spow}) z'(y)[(z(y)I-L_h)^{-1}g_h] \, dy. \quad\text{for } g\in \vh .
$$
Given a positive integer $N$ and a quadrature spacing $k>0$, we set $y_j := j k$ for $j = -N,...,N$ and define the sinc quadrature approximation of $e_{\gamma,1}(-t^\gamma L_h^\beta) g_h$ by
\beq\label{e:sincapp}
	Q_{h,k}^N(t)g_h:=\frac{k}{2\pi i}\sum_{j=-N}^{N} e_{\tpow,1}(-t^\tpow z(y_j)^\spow)z'(y_j)
[(z(y_j)I-L_h)^{-1}g_h] .
\eeq

\subsection{Quadrature Error}
We now discuss the quadrature error. 
Expanding $(E_h(t)-Q_{h,k}^N(t))g_h$ in term of the discrete eigenfunction $\{\psi_{j,h}\}_{j=1}^{M_h}$ (see Section~\ref{ss:discrete_op}), for $s>0$ we have
\begin{equation}\label{errorl}
\bal
\| (E_h(t)-Q_{h,k}^N(t))g_h\|^2_{\dH_h^{2s}} &= (2\pi )^{-2}\sum_{j=1}^\NUMDOF \lambda_{j,h}^{2s} |\cE(\lambda_{j,h},t)|^2
|(g_h,\psi_{j,h})|^2\\ & \le 
(2\pi)^{-2} \|g_h\|^2_{\dH^{2s}_h} \max_{j=1,\ldots,\NUMDOF} |\cE(\lambda_{j,h},t)|^2,
\eal
\end{equation}
where 
\beq\label{e:quaderr}
\cE({\lambda},t):=  \int_{-\infty}^{\infty} g_{\lambda}(y,t)\, dy -
 k\sum_{j=-N}^{N}g_{\lambda}(j k,t)
 \eeq
and
\begin{equation}\label{eq:lamdafun}
  g_\lambda(y,t):=  
e_{\tpow,1}(-t^\tpow z(y)^\spow)z'(y)
(z(y)-\lambda)^{-1}.
\end{equation}
The function $g_\lambda(y,t)$ is well defined for $t>0$, $\lambda \geq \lambda_1$, $y\in \CC$ with $z(y)\neq
\lambda$ and $z(y)$ not on the branch cut for the logarithm.   

Following \cite{LB92}, we show that when $k=c/\sqrt{N}$ for some constant $c$, the quantity $\cE(\lambda,t) \to 0$ when $k \to 0$ uniformly with respect to $\lambda \geq\lambda_1$.
Moreover, the convergence rate is $O(\exp{(-c\sqrt{N})})$. 
We then use this estimate in \eqref{errorl} to deduce exponential rate of convergence for the sinc quadrature scheme \eqref{e:sincapp}.

This program requires additional notations and we start with the class of functions $S(B_d)$.
\begin{definition}\label{class_SB}
Given $d>0$, we define the space $S(B_d)$ to be the set of functions $f$ defined on $\RR$ 
having the following properties:
\begin{enumerate}[(i)]
 \item $f$ extends to an analytic function in the infinite strip
 $$
 B_d:=\left\{z \in \mathbb C : \ \Im(z)< d\right\} 
 $$
 and is continuous on $\overline{B_d}$.
\item There exists a constant $C$ independent of $y\in \RR$ such that 
$$
\int_{-d}^{d} |f(y+iw)|\, dw\leq C;% \| \eta \|_{\widetilde H^\spow(D)} \|\widetilde \theta\|_{H^\spow(D)};
$$
\item We have 
$$N(B_d):=\int_{-\infty}^\infty \left(|f(y+id)|+|f(y-id)| \right) dy < \infty .$$
\end{enumerate}
\end{definition}
Note that condition $(ii)$ is more restrictive than actually needed (see Definition 2.12 in \cite{LB92}) but sufficient for our considerations.
In addition, For $f\in S(B_d)$,  Theorem 2.20 in \cite{LB92} provides the
error estimate for the quadrature approximation to $\int_{\mathbb R} f(x)\,dx$ using
an infinite number of equally spaced quadrature points with spacing   $k>0$:
\beq\label{infbound}
	\left|\int_{-\infty}^\infty  f(x)\, dx- k\sum_{j=-\infty}^{\infty}f(jk)\right|\le \frac{N(B_d)}{2\sinh(\pi d/k)}e^{-\pi d/k}.
\eeq

The lemma below is proved in Appendix~\ref{a:lemma2} and is the first
step in estimating the sinc quadrature error.

\begin{lemma}\label{l:contour-esti-i}
Let $\lambda \geq \lambda_1$ and $t>0$. The function $w \mapsto g_\lambda(w,t)$ belongs to $S(B_{d})$ for $0<d<\pi/4$.
Moreover, there exists a constant $C$ only depending on $\spow$, $d$ and $b$ such that
\beq\label{nd-bound}
\bal
	N(B_d)\le C(\spow,d,b)t^{-\tpow}.
\eal
\eeq
\end{lemma}

The above lemma together with the quadrature estimate \eqref{infbound}
leads to exponential decay for $\cE(\lambda,t)$ as provided in the
following lemma.

\begin{lemma}\label{l:sincquad}
Let $0<d<\pi/4$. There exists a constant $C$ only depending on $d$, $b$,
$\beta$ and $\lambda_1$ such that for $k<1$, $N>0$, $t>0$ and $\lambda
\geq \lambda_1$,  
\beq\label{quad-epsilon-bound}
|\cE(\lambda,t)|\le Ct^{-\tpow}\left(e^{-\pi d/k}+e^{-\spow Nk}\right) .
\eeq
\end{lemma}
\begin{proof}
In order to derived the desired estimate, we write
$$
\cE(\lambda,t) =  \left(\int_{-\infty}^\infty  g_\lambda (x,t)\, dx -k\sum_{j=-\infty}^\infty g_\lambda(jk,t)\right)+ k\sum_{|j| \geq N+1} g_\lambda(jk,t) . 
$$
Lemma~\ref{l:contour-esti-i} guarantees that $g_\lambda(.,t) \in S(B_d)$ and so in view of \eqref{infbound}, we obtain
$$
 	\left|\int_{-\infty}^\infty  g_\lambda (x,t)\, dx- k\sum_{j=-\infty}^{\infty}g_\lambda (jk,t)\right|\le \frac{N(B_d)}{2\sinh(\pi d/k)}e^{-\pi d/k} \leq C t^{-\gamma} e^{-\pi d/k},
$$
where $C$ is the constant in \eqref{nd-bound}.
For the truncation term, we use \eqref{e:estim_app} (in the appendix) to write
$$
	k\sum_{|j|\ge N+1} |g_\lambda(jk,t) |
	\leq
	C k\sum_{|j|\ge N+1} t^{-\gamma} e^{-\beta jk},
	$$
where $C$ is a constant only depending on $d$, $b$ and $\lambda_1$. 
Next we bound the infinite sum by the integral and arrive at
 $$
	k\sum_{|j|\ge N+1} | g_\lambda(jk,t)|  \leq C t^{-\tpow}e^{-\spow Nk},
$$
where now the constant depends on $\beta$ as well.
Gathering the above estimates completes the proof.
\end{proof}

\begin{remark}[Choice of $k$ and $N$]\label{r:kN}
The optimal combination of $k$ and $N$ is obtained by balancing the two exponentials on the right hand side of \eqref{quad-epsilon-bound}.
Hence, we select $k$ and $N$ such that $\pi d/k=\spow Nk$, i.e. $k=\sqrt{\frac{\pi d}{\spow N}}$, and
the estimate on $\cE(\lambda,t)$ becomes
\begin{equation}\label{e:sinc_quad_opt}
	|\cE(\lambda,t)|\le Ct^{-\tpow}e^{-\sqrt{\pi d\spow N}}.
\end{equation}
\end{remark}

Estimates on the difference between $E_h(t)$ defined by  \eqref{e:dsol1}
and $Q_{h,k}^N$ defined by \eqref{e:sincapp} follow from
\eqref{e:sinc_quad_opt} and  \eqref{errorl} as stated in the following theorem.

\begin{theorem}\label{l:hsincquad}
Let $s\in [0,1/2]$, $d\in (0,\pi/4)$, and let $N$ be a positive integer.  Set $k=\sqrt{\frac{\pi d}{\spow N}}$.
Then there exists a constant $C$ independent of $k$, $N$, $t$ and $h$ such that for every $g_h \in \dH_h^{2s}$
\begin{equation}\label{e:hsincquad}
	\|(E_h(t)-Q_{h,k}^N(t))g_h\|_{\dH_h^{2s}}\le Ct^{-\tpow}e^{-\sqrt{\pi d\spow N}}\|g_h\|_{\dH_h^{2s}}.
\end{equation}
\end{theorem}

\subsection{The Total Error} 
The discrete approximation after space and quadrature discretization is 
\begin{equation}\label{e:fully}
u_{h}^N(t) := Q_{h,k}^N \pi_h v, \quad\text{with } k =\sqrt{\frac{\pi d}{\spow N}}.
\end{equation}

Gathering the space and quadrature error estimates, we obtain the final estimate for the approximation of the homogeneous problem.

\begin{theorem}[Total error]\label{t:hterr}  Assume that the conditions
  of Theorem~\ref{l:semi-sp} and Theorem~\ref{l:hsincquad} hold. 
Then there exists a constant $C$ independent of $h$, $t$ and $N$ such that
$$
	\|u(t)-u_{h}^N(t)\|_{\HH^{2s}}\le D(t)h^{2\alal}\|v\|_{\HH^{2\delta}}+Ct^{-\tpow}e^{-\sqrt{\pi d\spow N}}\|v\|_{\HH^{2s}},
$$
provided the initial condition $v$ is in $ \HH^{2s}\cap \HH^{2\delta}$.
Here $D(t)$ is the constant given by \eqref{e:cd}.
\end{theorem}
\begin{proof}
We use the decomposition
$$
u(t) - u_h^N(t) = u(t) - u_h(t) + u_h(t) - u_h^N(t)
$$
and invoke Theorem~\ref{l:semi-sp}  with  $\mu=1$ and 
Lemma~\ref{l:hsincquad} with $g_h = \pi_h v$ to arrive at
$$
	\|u(t)-u_{h}^N(t)\|_{\dH^{2s}}\le D(t)h^{2\alal}\|v\|_{\dH^{2\delta}}+Ct^{-\tpow}e^{-\sqrt{\pi d\spow N}}\|\pi_h v\|_{\dH_h^{2s}}.
$$
The equivalence of norms \eqref{ineq:H_h_H} together with  stability of the $L^2$ projection \eqref{pih-bound} and  the equivalence property between the dotted spaces and interpolation spaces \eqref{e:interpolation_spaces} (see Proposition~\ref{p:equiv}) yield the desired result.
\end{proof}

\begin{remark}[Implementation]\label{r:implementation}
Denote $\widetilde U(t)$ the vector of coefficients of $u_h^N(t)$ with respect to 
the finite element local basis functions and
$\widetilde V$ the vector of inner product between $v$ and local basis functions. Let $\widetilde A$ and $\widetilde M$
be the stiffness and mass matrices.
Then
$$
	\widetilde U(t) = \frac{k}{2\pi i}\sum_{j=-N}^N e_{\gamma,1}(-t^\gamma z(y_j)^\beta)(z(y_j)\widetilde M+\widetilde A)^{-1}\widetilde V.
$$
\end{remark}
\begin{remark}[Complexity of the Implementation]\label{r:mat-aspect}
We take advantage of the exponential decay of the sinc quadrature by setting $N=c(\alal\ln(1/h))^2$
so that 
$$
	\|u(t)-u_{h}^N(t)\|_{\dH^{2s}}\le C\max(D(t),t^{-\tpow})h^{2\alal}.
$$
Hence,  computing $u_h^N(t)$ for a fixed $t$
requires $O( \log(1/h)^2)$ complex finite element system solves. 
\end{remark}

\subsection{Numerical Illustration}\label{s:num}
In this section, we provide numerical illustrations of the rate of
convergence predicted by Theorem~\ref{l:semi-sp} and Lemma~\ref{l:hsincquad}.

\subsubsection*{Space Discretization Error} 

In order to illustrate the space discretization error, we start with a one dimensional problem and use a spectral decomposition to compute the exact solution without resorting to quadrature. 
Set $\Omega=(0,1)$ , $Lu:=-u^{\prime\prime}$.
We chose the initial condition to be
$v \equiv 1$ or, using the eigenvalues $\lambda_{\ell} = \pi^2 \ell^2$ and associated eigenfunctions $\psi_\ell(x) = \sqrt{2}\sin(\pi \ell x)$, 
$$
v = 2\sum_{\ell=1}^\infty \frac{1-(-1)^\ell}{\pi \ell}  \sin(\pi\ell x) \approx 2\sum_{\ell=1}^{50000} \frac{1-(-1)^\ell}{\pi \ell}  \sin(\pi\ell x).
$$
The number of term used before the truncation is chosen large enough not to influence the space discretization ($50000$).
With these notations, the exact solution for $\gamma=1/2$ and $0<\beta<1$ is approximated by
\begin{equation}\label{e:u_exact_1d}
u(t) \approx 2\sum_{\ell=1}^{50000} e_{1/2,1}(-t^{1/2} (\pi\ell)^\beta)  \frac{1-(-1)^\ell}{\pi \ell}  \sin(\pi \ell x).
\end{equation}

For the space discretization, we consider a sequence of uniform meshes with  mesh sizes $h_j=2^{-j}$, where $j=1,2,\dots$
and denote by  $\{\varphi_{k,h}\}_{k=1,\ldots,M_{h_j}}$ the continuous piecewise linear finite element basis of $\vh$.
The eigenvalues of $L_{h_j}$ corresponds to the eigenvalues of $M_{h_j}^{-1}S_{h_j}$, where $M_{h_j}$ and $S_{h_j}$ are the mass and stiffness matrices and are given by
$$
\lambda_{\ell,h_j}= \frac{6(1+\cos(k\pi h_j))}{h_j^2 (2+\cos(k \pi h_j))}.
$$
The associated eigenfunctions to $L_h$ are
$$
\psi_{\ell,h_j} := \sum_{k=1}^{M_{h_j}}\sqrt{2 h_j} \sin( h_j \ell k \pi) \varphi_{k,h_j}.
$$
Similar to \eqref{e:u_exact_1d}, we use the discrete spectral representation below of  $u_{h_j}(t)$ for our computation
$$
u_{h_j}(t) = \sum_{\ell=1}^{M_{h_j}} e_{1/2,1}(-t^{1/2} \lambda_{\ell,h_j}^\beta)  v_{h_j,\ell}  \psi_{\ell,h_j},
$$
with 
$$v_{\ell,h_j} = \int_0^1 \psi_{\ell,h_j}(x)\, dx =  h_j \sqrt{2h_j}\sum_{k=1}^{M_{h_j}} \sin(h_j \ell k \pi) .$$

%Note that the resulting stiffness and mass matrices can be diagonalized by
%the discrete sine transform, whose matrix representation is given by 
%$S_j=\{s^j_{kl}\}$ with $s^j_{kl}=\sqrt{2h_j}\sin{(kl\pi h)}$. 
%Moreover, the eigenvalues of stiffness and mass matrices are 
%$a^j_k=(2+2\cos{(k\pi h_j)})/h_j$ and $m^j_k=h_j(4+2\cos{(k\pi h_j)})/6$ for $k=1,2,\ldots,M_{h_j}$. The eigenvalue of $L_h$
%is $\lambda^j_k=a^j_k/m^j_k$ and hence the eigenvalue of the semi-discrete approximation $u_{h_j}(t)$ is given by 
%$\Lambda^j_k(t)=e_{\tpow,1}(-t^\tpow (\lambda^j_k)^\spow)$. Now let $V_j$ be the vector of products between the initial data $v$
%and the finite element basis functions $\psi_{k,h_j}$, i.e. the $k$th component of $V_j$ is $(v,\psi_{k,h_j})$. Then, the vector of 
%coefficients of $u_{h_j}(t)$ can represented by 
%$S_j^{-1}\Lambda_j(t) S_j V_j$,
%where $\Lambda_j(t)$ is a diagonal matrix whose diagonal entries are $\Lambda^j_k(t)$ for $k=1,\cdots,M_{h_j}$. 

Note that $\alpha$ in Assumption~\ref{regularity} is 1,  $v\in \dH^{1/2-\epsilon}$ for any $\epsilon>0$ so that $\delta=1/4-\epsilon$.
The error will be computed in $L^2$ and $H^1$, i.e. $s=0$ and $s=1/2$.
For the latter we need $\beta > 1/4$.   
The predicted convergence rates (Theorem~\ref{l:semi-sp}) are 
$$
2\alal=1+\min(1,1-2s,2(\beta+\delta-s)-1-\epsilon)
$$
for every $\epsilon>0$, i.e.
$$
\|u(t)-u_h(t)\|+h\|u(t)-u_h(t)\|_{H^1}\le D(t)h^{\min(2,2\beta+1/2)-\epsilon} \| v \|_{\dH^{1/2-\epsilon}}.
$$

We use the \textit{MATLAB} code \cite{mlcode} to evaluate $e_{\gamma,1}(z)$ for any $z\in\mathbb C$ and fix $t=0.5$.
In Figure~\ref{f:hsp}, we report the errors 
$$e_j:=\|u(t)-u_{h_j}(t)\|\quad\text{and}\quad e^1_j:=\|u'(t)-u'_{h_j}(t)\|$$
 for $j=3,4,5,6,7$ and different values of $\spow$. 
The observed rate of convergence 
$$OROC:=\frac{\ln(e_7/e_6)}{\ln{2}}\quad\text{and}\quad OROC^1:=\frac{\ln(e^1_7/e^1_6)}{\ln{2}}$$
are also reported
in this figure and match the rates predicted by  Theorem~\ref{l:semi-sp}.

\begin{figure}[hbt!]
 \begin{center}
    \begin{tabular}{cc}
\includegraphics[scale=.47]{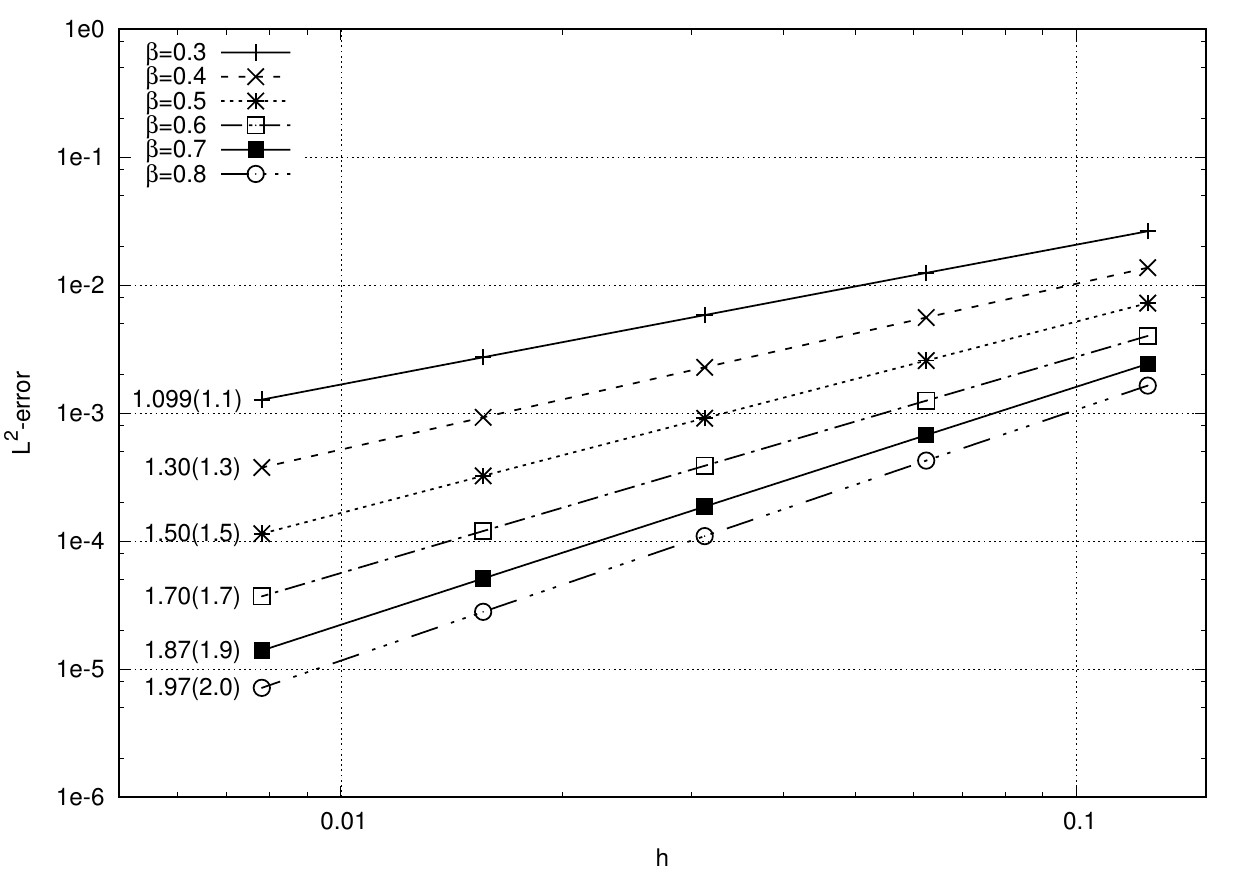} & \includegraphics[scale=.47]{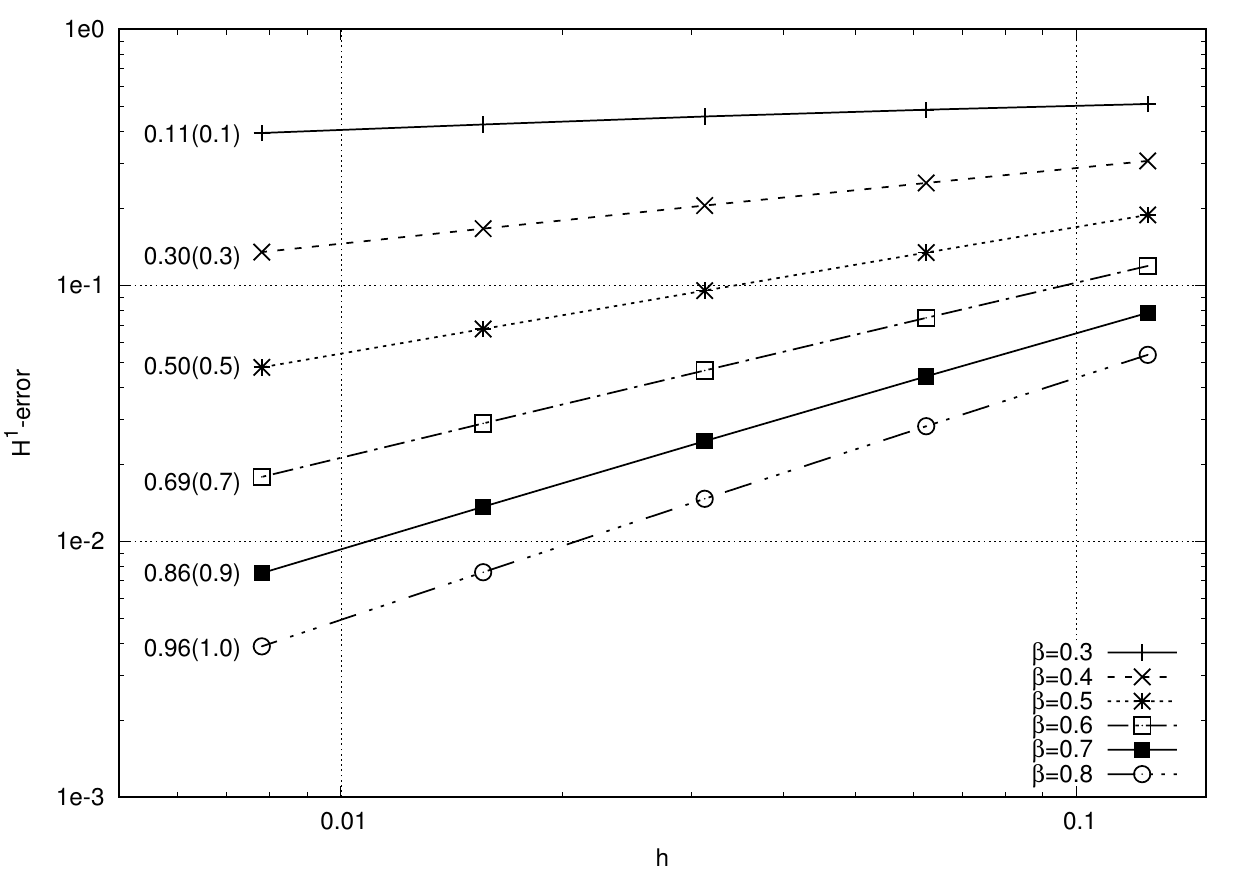}\\
    \end{tabular}
 \end{center}
    \caption{Errors $e_j$ (left) and $e^1_j$ (right) versus the mesh size $h$ for different values of $\beta$.
    The observed rate of convergence $OROC$ and $OROC^1$ are reported on
    the left of each graph and match the rate predicted by 
Theorem~\ref{l:semi-sp} shown in between  parentheses.}
    \label{f:hsp}
\end{figure}
    
\subsubsection*{Effect of the Sinc Quadrature}

We examine the error between the semi-discrete approximation and its sinc quadrature approximation.
To this end and in order to factor out the space discretization, it suffices to observe $\cE({\lambda},t)$ defined by \eqref{e:quaderr} for all 
$\lambda \geq \lambda_1$. Here we fix $t=0.5$ and approximate $\|\cE(.,t)\|_{L^\infty(\lambda_1,\infty)}$ 
with $\lambda_1=10$ using the method discussed in Section 5.2 in \cite{BLP17}.  
For the hyperbolic contour $z(y)$ in \eqref{hc}, we choose $b=1$ so that $b\in (0,\lambda_1/\sqrt{2})$.
Following Remark~\ref{r:kN}, we fix the number of quadrature points to $2N+1$ and balance the two source of errors by setting $k=\sqrt{{\pi d}/{(\spow N)}}$ with $d={\pi}/{8}$. 
According to \eqref{e:sinc_quad_opt}, we have
$$
\|\cE(\lambda,t)\|_{L^\infty(10,\infty)} \leq C t^{-\gamma} e^{-\sqrt{\pi d \beta N}}.
$$
The left graph of Figure~\ref{f:hsinc} illustrates the exponential decay of $\|\cE(\lambda,t)\|_{L^\infty(10,\infty)}$ as
$N$ increases for $\tpow=0.5$ and $\spow=0.3,0.5,0.7$. We also report (right) the singular behavior of $\|\cE(.,t)\|_{L^\infty(10,\infty)}$
in time for $N=100$, $\spow=0.5$ and $\tpow=0.3,0.5,0.7$. 

\begin{figure}[hbt!]
 \begin{center}
    \begin{tabular}{cc}
\includegraphics[scale=.31]{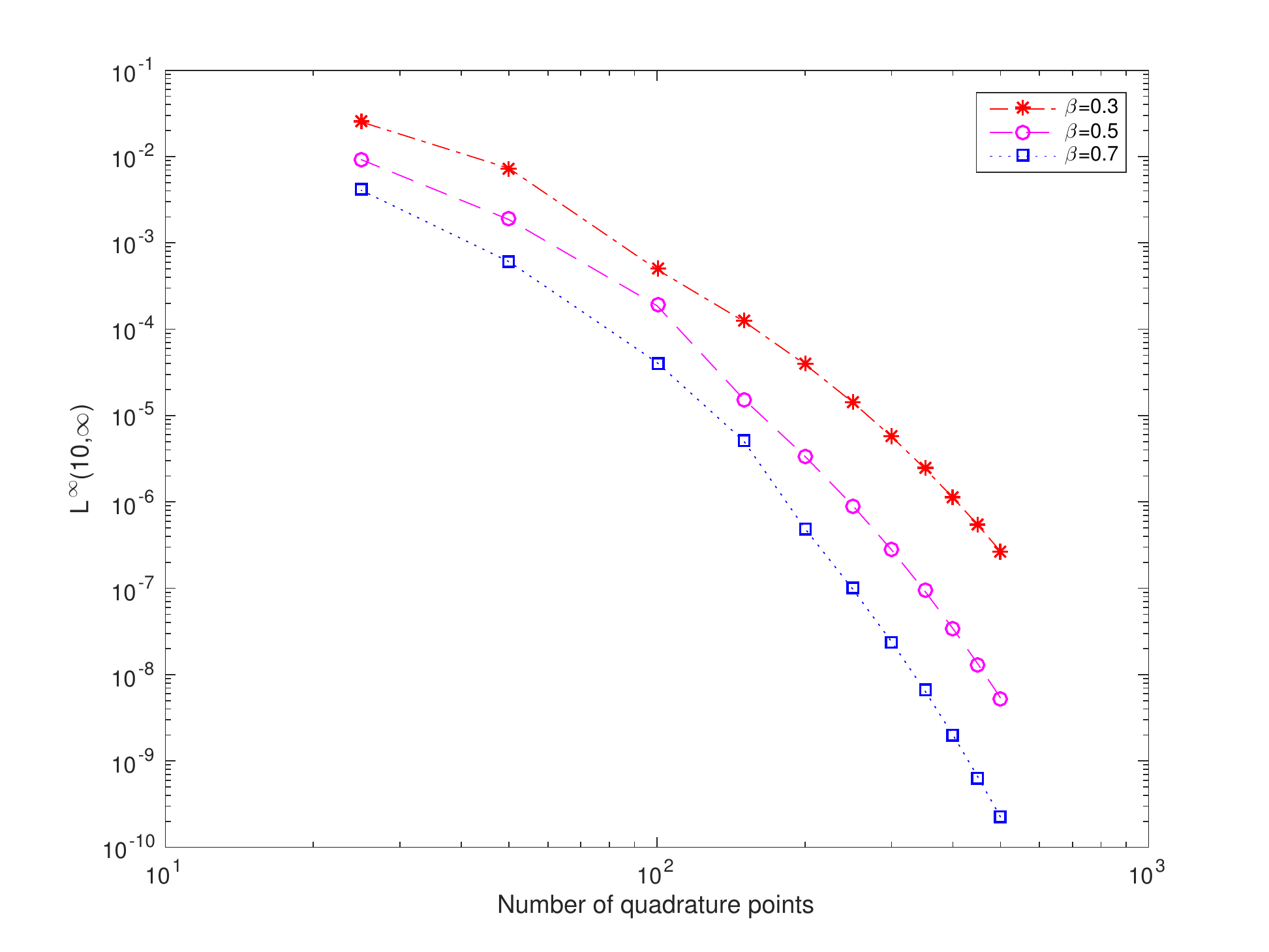} & \includegraphics[scale=.31]{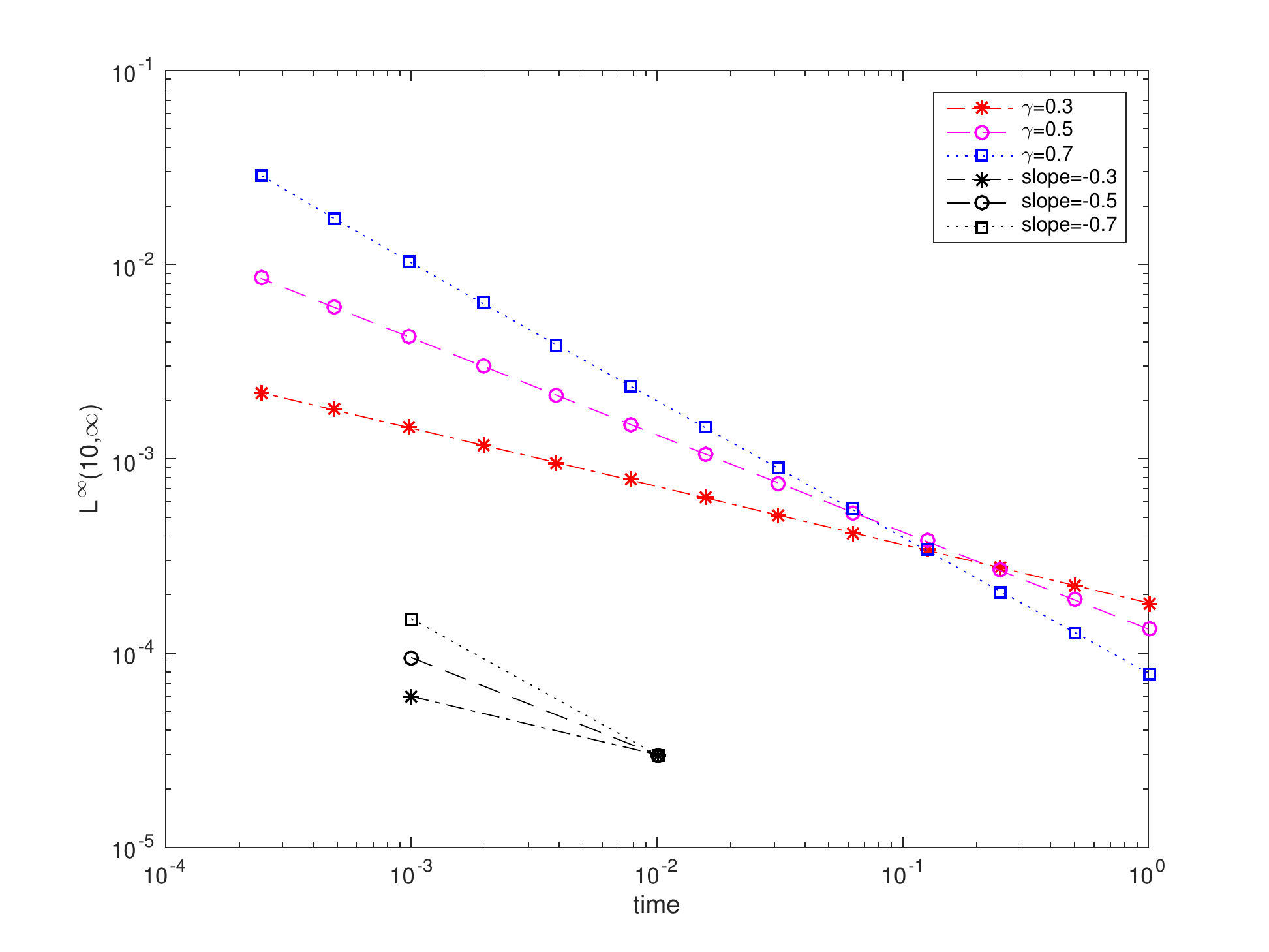}\\
    \end{tabular}
 \end{center}
    \caption{(Left) Exponential decay of $\| \cE({.},0.5)\|_{L^\infty(10,\infty)}$ versus the number of quadrature points used for different values of $\beta$.
    (Right) Singular behavior of $\| \cE(.,t)\|_{L^\infty(10,\infty)}$ as $t \to 0$ for $\beta=0.5$ and different values of $\gamma$. The rate $-\tpow$ predicted by \eqref{e:sinc_quad_opt} is observed.}
    \label{f:hsinc}
\end{figure}

\subsubsection*{A Two Dimensional Problem}
We now focus our attention to the total error in a two dimensional problem. 
Let $\Omega=(0,1)^2$, $L=-\Delta$ and the initial condition be the eigenfunction of $L$ given by 
$$
v(x_1,x_2)=\sin{(\pi x_1)}\sin(\pi x_2).
$$ 
The exact solution is then given by 
$$u(t,x_1,x_2)=e_{\tpow,1}(-t^\tpow (2\pi^2)^\spow)\sin{(\pi x_1)}\sin{(\pi x_2)} .$$ 
The space discretizations are subordinate to a sequence of uniform subdivisions made of triangles with the mesh size $h_j=2^{-j}\sqrt{2}$. 
For the quadrature, we chose $N=400$ and set $k= \sqrt{{\pi^2}/{(8\spow N)}}$ for the quadrature error not to affect the space discretization error.
Since $\lambda_1=\pi^2$, we again set $b=1$ in \eqref{hc}.
We fix $t=0.5$, $\gamma=0.5$ and report in Figure~\ref{f:2Dh}, the quantities $\|u(t)-u_{h_j}^N(t)\|$ for $j=3,4,5,6,7,8$ and different $\spow$.
As announced in Theorem~\ref{t:hterr}, a second order rate of convergence is observed.
\begin{figure}[hbt!]
\begin{center}
\includegraphics[scale=.47]{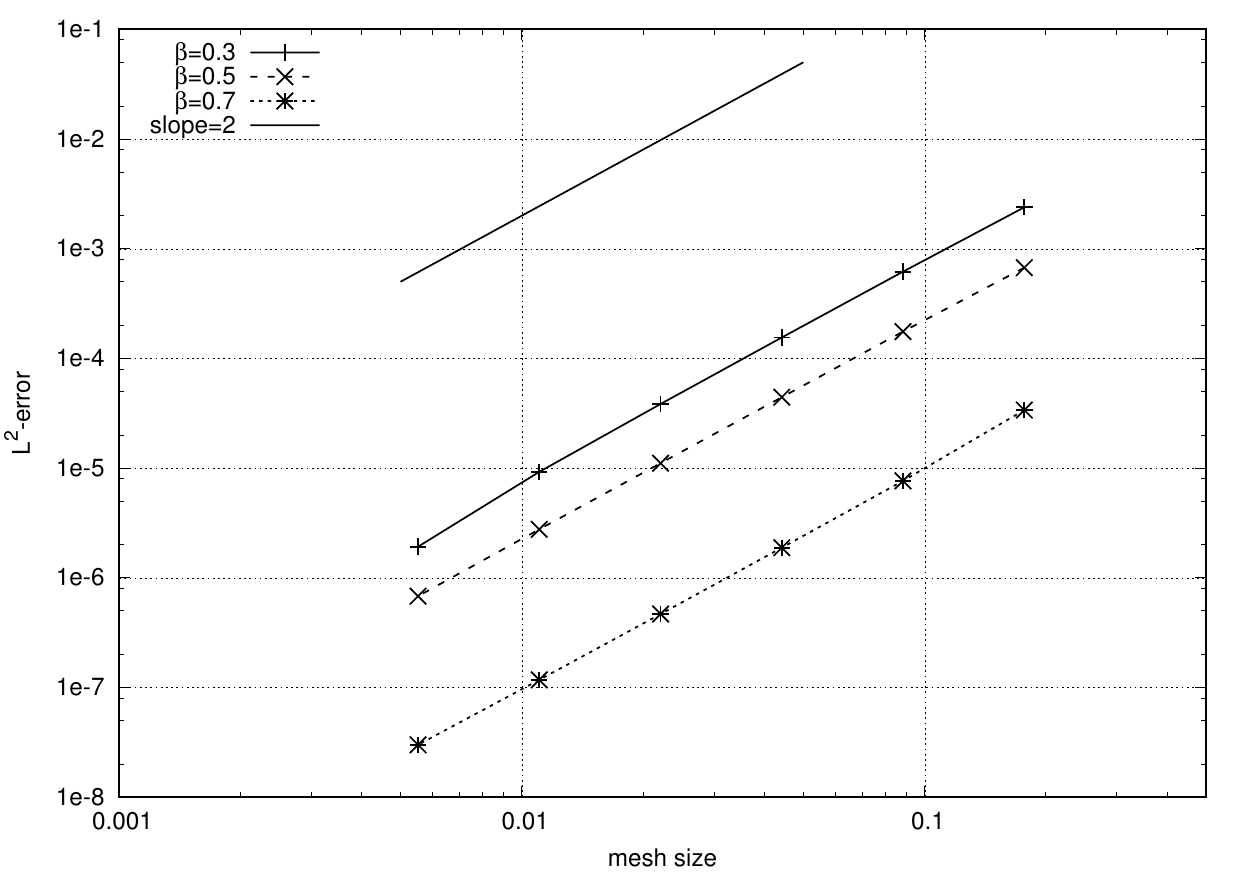} 
\end{center}
\caption{$L^2$ error between $u(0.5)$ and $u_{h_j}^N(0.5)$ with $\tpow=0.5$ and different values of $\spow$. 
A second order convergence rate is observed.}
\label{f:2Dh}
\end{figure}
%%%%%%%%%%%%%%%%%%%%%%%%%%%%%%%%%%%%%%%%%%%%%%%%%%

\section{Approximation of the Non-homogeneous Problem}\label{s:nh}

We now turn our attention to the non-homogeneous problem, i.e. $f\neq 0 $ and $v=0$ in \eqref{e:p},  for which the solution reads
\begin{equation}\label{e:exact_nh}
u(t) = \int_{0}^t \underbrace{r^{\tpow-1}e_{\tpow,\tpow}(-r^\tpow L^\spow)}_{=:W(s)} f(t-r)\, dr.
\end{equation}

\subsection{The Semi-discrete Scheme}

According to \eqref{e:dsol1}, the finite element approximation of \eqref{e:exact_nh} is given by  
\beq\label{e:discrete_nh}
	u_h(t)=\int_0^t \underbrace{r^{\gamma-1} e_{\gamma,\gamma}(-r^\gamma L_h^\spow)}_{=:W_h(r)} \pih f(t-r)\, dr.
\eeq

As in the homogeneous case, the finite element approximation error is
derived from Lemma~\ref{l:semi-sp} and we have the following lemma.
\begin{lemma}[Space Discretization for the non-homogeneous problem]\label{l:semirhs}
Assume that Assumption~\ref{regularity} holds for  $\alpha\in (0,1]$.
Let $\gamma \in (0,1)$, $s \in [0,\frac 1 2]$ and let $\alpha^*$ and $\delta$
be as in \eqref{e:astar} and \eqref{delta}, respectively.
There exists a constant $C$ such that 
 $$
 \|u(t)-u_h (t)\|_{\dH^{2s}}\le   \widetilde D(t) h^{2\alal} \|f\|_{L^{\infty}(0,t;\dot{H}^{2\delta})},
 $$
 where 
 \begin{equation}\label{e:tildeD}
 \widetilde D(t) =  C  
 \left\lbrace
\begin{array}{ll}
t^\tpow & \qquad \text{when }\delta > \alpha^*+s, \\
t^\tpow \max(1,\ln(1/t))  & \qquad \text{when }\delta = \alpha^*+s, \\
t^{\tpow-\tpow(\alal +s- \delta)/\beta} & \qquad \text{when }\delta < \alpha^*+s.
\end{array}
\right.
\end{equation}
\end{lemma}
\begin{proof}
Applying Theorem~\ref{l:semi-sp} gives
\begin{align*}
	\|u(t)-u_h(t)\|_{\dH^{2s}}&\le \int_0^t {r}^{\tpow-1} \|e_{\tpow,\tpow}(-t^\tpow L^\spow)-e_{\tpow,\tpow}(-t^\tpow L_h^\spow)\pih\|_{\dH^{2\delta}\rightarrow \dH^{2s}}
	\|f(t-{r})\|_{\dH^{2\delta}}\, d{r} \\
	&\le Ch^{2\alal}\|f\|_{L^\infty(0,t;\dH^{2\delta})} \int_0^t {r}^{\tpow-1} D({r})\, d{r},
\end{align*}
where $D(t)$ is given by \eqref{e:cd}.
The conclusion follow from  $\int_{0}^t r^{\gamma-1} D(r)\, dr =
\widetilde D(t)$.
\end{proof}

\subsection{Time Discretization via Numerical Integration}

Given a final time $\tt$, we discuss first a numerical approximation of the integral
$$
\int_0^\tt W_h(s) \pi_h f(\tt-s)\, ds.
$$
For simplicity, we set  
$$g(s)= f(\tt-s)$$
 so that the above integral
becomes
 $$
\int_0^\tt W_h(s) \pi_h g(s)\, ds.
$$
For a positive integer $\cM$, let $0=t_0<t_1<...<t_{\cM}=\tt$ be a partition of the time interval $[0,\tt]$. 
On each subinterval we set $t_{j-\frac12} = \frac 1 2 (t_j+t_{j-1})$ and
propose the pseudo-midpoint approximation
\begin{equation} \label{e:quad_int}
\begin{split}
&\int_{t_{j-1}}^{t_{j}} W_h(r) \pi_h g(r) \, dr \\
&\qquad \approx \int_{t_{j-1}}^{t_{j}} W_h(r)\, dr\,  \pi_h g(t_{j-\frac12}) \\
& \qquad \ = L_h^{-\beta} \left( e_{\gamma,1}(-t_{j-1}^\gamma L_h^\beta) - e_{\gamma,1}(-t_{j}^\gamma L_h^\beta) \right)\pi_h g(t_{j-\frac12}),
\end{split}
\end{equation}
where to achieve the last step,  we used relation \eqref{e:fd2}.

Before going further, we note that numerical methods based on
\eqref{e:quad_int} cannot perform optimally when using a  uniform decomposition of the time interval because $W_h(t)$ is singular at $t=0$.
Hence, the performance of algorithms based on uniform partitions are bound to the error on the first interval $(0,t_1)$.
Measuring in the $\dH^{2s}$-norm for $s\in [0,1/2]$, we have
\begin{equation}\label{e:rate_first_uniform}
\begin{split}
& \bigg\| \int_0^{t_1} W_h(r) \pi_h (g(r)-g(t_{1/2}))\, dr\bigg\|_{\dH^{2s}}  \\
&\qquad \le C \int_0^{t_1} \|W_h(r)\|_{\dH^{2s}\rightarrow \dH^{2s}} \|g(r)-g(t_{1/2})\|_{\dH^{2s}}\, dr \\
&\qquad \leq C t_1 \|f_t\|_{L^\infty(0,T;\dH^{2s})} \int_0^{t_1} r^{\gamma-1}\, dr  \leq C t_1^{1+\gamma} \|f_t\|_{L^\infty(0,T;\dH^{2s})}.
\end{split}
\end{equation}

To overcome this deterioration, we propose a geometric refinement of the
partition near $t_0=0$ which depends on two positive integers $\cM$ and
$\cN$ (see also Section 3.1 of \cite{BP15}).   We first set
$$
t_j := 2^{-(\cM-j)} \tt, \qquad j=1,...,\cM.
$$
We decompose further all but the first interval
$$
I_j:= [t_{j},t_{j+1}] = [ 2^{-(\cM-j)}\tt, 2^{-(\cM-j-1)}\tt], \qquad j=1,\ldots,\cM-1
$$
onto $\cN$ subintervals 
$$
t_{j} = t_{j,0} < ... < t_{j,l} < ... < t_{j,\cN} = t_{j+1} 
$$
where, for $l=0,...,\cN$,
\begin{equation}\label{e:tjl}
t_{j,l} := t_{j} + l \tau_j , \qquad \textrm{with }\tau_j :=  |I_j|/\cN = 2^{-(\cM-j)} \tt/\cN.
\end{equation}

As in \eqref{e:quad_int}, we approximate
$$
\int_{t_{j,l-1}}^{t_{j,l}} W_h(r) \pi_h g(r)\, dr
$$
on each subinterval $I_{j,l}:= [ t_{j,l-1},t_{j,l}]$ by
\begin{equation}\label{e:rel_num_nh}
L_h^{-\spow}\left( e_{\tpow,1}(-{t}_{j,l-1}^\tpow L_h^\spow)-e_{\tpow,1}(-{t}_{j,l}^\tpow L_h^\spow) \right)\pih g({t}_{j,l-1/2}).
\end{equation} 
Here $t_{j,l-1/2}:=\frac 1 2 (t_{j,l-1}+t_{j,l})$.
We  use the bar symbol to denote average quantities over the interval $[t_{j,l-1},t_{j,l}]$, e.g.,
$$
\bE_{j,l} : \vh \rightarrow \vh, \qquad \bE_{j,l}:=\frac{1}{\tau_j}\int^{{t}_{j,l}}_{{t}_{j,l-1}} W_h(r)\, dr.
$$ 
and
$$
\bg_{j,l}:=\frac{1}{\tau_j}\int_{t_{j,l-1}}^{t_{j,l}}  g(r)\, dr.
$$

The approximate solution after time integration is thus given by
\begin{equation}\label{e:space_time_solution}
u_h^{\cN,\cM}(\tt) :=\sum_{j=1}^{\cM-1}\tau_j \sum_{l=1}^\cN \bE_{j,l} (\pi_h f(\tt-t_{j,l-\frac12})).
\end{equation}

We start by assessing the local integration error 
$$
\int_{t_{j,l-1}}^{t_{j,l}} W_h(r) \pi_h (f(\tt-r)-f(\tt-t_{j,l-1/2}))\, dr.
$$
\begin{lemma}[Local Approximation]\label{l:time_stepping_local}
Let $\gamma \in (0,1)$ and $s\in [0,1/2]$.
Let $j\geq 2$ and assume that $g(t)=f(\tt-t)$ belongs to  $H^2(t_{j-1},t_j;\dH^{2s})$. 
There exists a constant $C$ independent of $h$, and $\tau_j$ such that on every interval $I_{j}=[t_{j-1},t_j]$, we have
\begin{equation*}
\begin{split}
& \| \sum_{l=1}^{\cN}  \int_{t_{j,l-1}}^{t_{j,l}} W_h(r) \pi_h (g(r)-g(t_{j,l-1/2}))~dr \|_{\dH^{2s}} \\
& \qquad \leq C \tau_j^{5/2} \left( \sum_{l=0}^\cN t_{j,l}^{2\gamma-2} \right)^{1/2} \| g_{tt} \|_{L^2(t_{j-1},t_j;\dH^{2s})} 
+ C \tau_j^3 \left( \sum_{l=0}^\cN t_{j,l}^{\gamma-2}\right) \| g_t \|_{L^\infty(t_{j-1},t_j;\dH^{2s})}.
\end{split}
\end{equation*}
\end{lemma}
\begin{proof} We use  the following decomposition on each sub-interval:
$$
\bal
& \int_{t_{j,l-1}}^{t_{j,l}} W_h(r) \pi_h (g(r)-g(t_{j,l-1/2}))\, dr\\
& \qquad 
	 =\underbrace{\dt_j \bE_{j,l}\pi_h (\bg_{j,l}-g({t}_{j,l-\frac12}))}_{=:E_1}  + \underbrace{\int_{{t}_{j,l-1}}^{{t}_{j,l}} (W_h({r})-\bE_{j,l})\pi_h(g({r})-g({t}_{j,l-\frac12}))\, d{r}}_{=:E_2}.
	 \eal
$$

$\boxed{1}$ We estimate $E_1$
\beq\label{e2-bound}
\|E_1\|_{\dH^{2s}} \le \dt_j\|\bE_{j,l}\pi_h\|_{\dH^{2s} \to \dH^{2s}}\|\bg_{j,l}-g({t}_{j,l-\frac12})\|_{\dH^{2s}} .
\eeq
We now bound $\|\bE_{j,l}\pi_h\|_{\dH^{2s} \to \dH^{2s}}$ and $\|\bg_{j,l}-g({t}_{j,l-\frac12})\|_{\dH^{2s}} $ separately. 
For the latter, we  expand $g(\eta)$ at $\eta=t_{j,l-\frac12}$ to get
$$
g(\eta)-g({t}_{j,l-\frac12})=(\eta-{t}_{j,l-\frac12})g_t({t}_{j,l-\frac12})+\int_{t_{j,l-\frac12}}^{\eta}(r-{t}_{j,l-\frac12}) g_{tt}(r)\, dr ,
$$
where $g_t$ and $g_{tt}$ denote the first and second partial derivative in time of $g$. 
As a consequence, taking advantage of $t_{j,l-\frac12}$ being the midpoint of the interval $I_{j,l}$, we obtain
\begin{align*}
\bg_{j,l}-g({t}_{j,l-\frac12})  &= \frac{1}{\dt_j} \int_{{t}_{j,l-1}}^{{t}_{j,l}}\left(g(\eta)-g({t}_{j,l-\frac12}) \right)\, d\eta \\
&=\frac{1}{\dt_j} \int_{{t}_{j,l-1}}^{{t}_{j,l}} \int_{t_{j,l-\frac12}}^{\eta}(r-{t}_{j,l-\frac12}) g_{tt}(r)\, dr\, d\eta 
\end{align*}
and so using a Cauchy-Schwarz inequality
\begin{equation}\label{gt-bound}
\| \bg_j-g({t}_{j-\frac12})\|_{\dH^{2s}} \leq  \tau_j^{3/2} \| g_{tt}\|_{L^2(t_{j,l-1},t_{j,l};\dH^{2s})}.
\end{equation}

In order to bound  $\|\bE_{j,l}\pi_h\|_{\dH^{2s} \to \dH^{2s}}$, we note that from the definition of the discrete dotted spaces $\dH^{2s}_h$ (see \eqref{e:dotted_discrete_norm}), we have
$$
\|e_{\tpow,\tpow}(-t^\tpow L_h^\spow)\|_{\dH^{2s}_h \to \dH^{2s}_h}\le C.
$$
Therefore, from the expression of $W_h(t)$ in \eqref{e:discrete_nh}, the equivalence of norms \eqref{ineq:H_h_H} and the stability estimate \eqref{pih-bound} for $\pi_h$, we derive that
\beq\label{bw-bound}
\bal
\|\bE_{j,l}\pi_h\|_{\dH^{2s}\to \dH^{2s}}&\le \frac{1}{\dt_j}\int_{{t}_{j,l-1}}^{{t}_{j,l}}\eta^{\tpow-1}\|e_{\tpow,\tpow}(-\eta^\tpow L_h^\spow) \pi_h\|_{\dH^{2s}_h \to \dH^{2s}_h}\, d\eta\\
&\le \frac{C}{\dt_j}\int_{{t}_{j,l-1}}^{{t}_{j,l}}\eta^{\tpow-1}\, d\eta\le C {t}_{j,l-1}^{\tpow-1} .
\eal
\eeq

Estimates \eqref{gt-bound} and \eqref{bw-bound}  into \eqref{e2-bound} give the final bound for $E_1$
\beq\label{e2-last-bound}
	\|E_1\|_{\dH^{2s}}\le C\dt_j^\frac52  {t}_{j,l-1}^{\tpow-1} \|g_{tt}\|_{L^2({t}_{j,l-1},{t}_{j,l};\dH^{2s})}.
\eeq

$\boxed{2}$ We estimate $E_2$
\beq \label{e:E2}
\|E_2\|_{\dH^{2s}} \le  \int_{{t}_{j,l-1}}^{{t}_{j,l}}\|(W_h({r})-\bE_{j,l})\pi_h\|_{\dH^{2s}\to \dH^{2s}}\|g({r})-g({t}_{j,l-\frac12})\|_{\dH^{2s}}\, d{r} .\eeq
In this case as well, we need to estimate two terms separately, namely $\|(W_h({r})-\bE_{j,l})\pi_h\|_{\dH^{2s}\to \dH^{2s}}$ and $\|g({r})-g({t}_{j,l-\frac12})\|_{\dH^{2s}}$. 
For the latter, we write
\beq\label{gj-bound}
	\|g({r})-g({t}_{j,l-\frac12})\|_{\dH^{2s}} = \|  \int_{t_{j-\frac12}}^r g_t(\eta)\, d\eta \|_{\dH^{2s}} \leq \tau_j \|g_t\|_{L^\infty(t_{j,l-1},t_{j,l};
	\dH^{2s})}
\eeq

Next, we bound $\|(W_h({r})-\bE_{j,l})\pi_h\|_{\dH^{2s} \to \dH^{2s}}$. As before, it suffices to estimate
$\|W_h({r})-\bE_{j,l}\|_{\dH^{2s}_h \to \dH^{2s}_h}$.
To achieve this, we use the eigenfunctions $\{\psi_{i,h}\}_{i=1}^{M_h}$ of $L_h$.
By \eqref{e:fd2_2},
\begin{equation*}
\begin{split}
 W_h'({r})\psi_{i,h}  =   &  r^{\tpow-2} \lbrace (\tpow-1)e_{\tpow,\tpow}(-r^\tpow \lambda_{i,h}^\spow)\\
  & +r^\tpow \lambda_{i,h}^\beta    ((\tpow-1)e_{\tpow,2\tpow}(-r^\tpow \lambda_{i,h}^\beta)-e_{\tpow,2\tpow-1}(-r^\tpow \lambda_{i,h}^\spow))\rbrace \psi_{i,h}.
\end{split}
\end{equation*}
This  and  \eqref{ml-bound-scalar} with $z=-r^\gamma\lambda_{i,h}^\beta$ imply that  for $r \in I_{j,l}$,
$$
\|  W_h'({r})\psi_{i,h} \|\leq C r^{\gamma-2} \leq C t_{j,l-1}^{\gamma-2},
$$
where the constant in the above inequality is independent of $j$, $l$ and $h$.
Whence, $$\|W'_h(r)\|_{\dot H^{2s}_h\to \dot H^{2s}_h}\leq C t_{j,l-1}^{\gamma-2}$$ and 
$$
\|W_h({r})-\bE_{j,l}\|_{\dH^{2s}_h \to \dH^{2s}_h} \leq C \tau_j  \sup_{r \in I_{j,l}}\|W_h'(r) \|_{\dH^{2s}_h \to \dH^{2s}_h} \leq C \tau_j  t_{j,l-1}^{\gamma-2}.
$$
The above estimate and \eqref{gj-bound} in \eqref{e:E2} yield the final bound on $E_2$
\beq\label{e3-last-bound}
	\|E_2\|_{\dH^{2s}_h} \le C\dt_j^3   t_{j,l-1}^{\gamma-2} \|g_t\|_{L^\infty(t_{j,l-1},t_{j,l};\dH^{2s})}.
\eeq

\boxed{3} Summing up the contribution from each subinterval and using a Cauchy-Schwarz inequality, yields the desired result.
\end{proof}

\begin{remark}[Uniform time-stepping]\label{r:uniform_step}
In the case of uniform time-stepping, i.e. $\cN=0$ and $t_j = j \tau$,
$\tau = \tt/\cM $, we derive from the estimate provided in
Lemma~\ref{l:time_stepping_local} and the first interval estimate
\eqref{e:rate_first_uniform} that the quadrature error behaves
asymptotically like $\tau^{1+\gamma}$. We do not pursue  this further
but  rather investigate errors coming from the geometric partition. 
\end{remark}

\begin{theorem}[Time Discretization of Non-Homogeneous Problem]\label{t:geo}
Let $\gamma \in (0,1)$, $s\in [0,1/2]$, $\tt \geq\tt_0 >0$. Let $\cN$ be a positive
integer and  
\beq
\cM=\left\lceil\frac{2\log_2{\cN}}{\tpow}\right\rceil.
\label{cM}
\eeq
Assume that $f$ is in $H^2(0,\tt;\dH^{2s})$ and let $u_h^{\cN}(\tt):=
u_h^{\cN,\cM}$ be defined by \eqref{e:space_time_solution} and let $u_h(\tt)$ be the semi-discrete in space solution \eqref{e:discrete_nh}. Then there exists a constant $C$ independent of $\cN$, $h$ and $\tt$ satisfying
$$
	\|u_h(\tt)-u_h^{\cN}(\tt)\|_{\dH^{2s}} \le C\max(\tt^\gamma,\tt^{\frac32+\gamma})\cN^{-2}  \|f\|_{H^2(0,\tt;\dH^{2s})}.
$$
\end{theorem}
\begin{proof}
Using the definitions of $u_h(\tt)$ and $u_h^{\cN}(\tt)$, we write
\begin{equation*}
\begin{split}
u_h(\tt)-u_h^{\cN}(\tt) &= \int_0^{t_1} W_h(r) \pi_h f(\tt-r)\, dr \\
& \quad + \sum_{j=1}^{\cM-1} \sum_{l=1}^\cN  \int_{t_{j,l-1}}^{t_{j,l}} W_h(r) \pi_h (f(\tt-r)-f(\tt-t_{j,l-1/2}))\, dr.
\end{split}
\end{equation*}
For the first term, we note that \eqref{ml-bound-scalar} immediately
implies that $\| e_{\gamma,\gamma}(-r^\gamma L_h^\beta) \|_{\dH^{2s}\to \dH^{2s}} \leq C$.
The stability of the $L^2$ projection \eqref{pih-bound}  and
\eqref{cM} give
\begin{align*}
\bigg\| \int_0^{t_1} W_h(r) \pi_h f(\tt-r)\, dr \bigg \|_{\dH^{2s}}& \leq C
                                                               2^{-\gamma(\cM-1)}\tt^{\gamma}
                                                               \| f \|_{L^\infty(0,\tt;\dH^{2s})}\\
&\leq C \tt^\gamma \cN^{-2}  \| f\|_{L^\infty(0,\tt;\dH^{2s})}.  
\end{align*}
For the second term, we apply Lemma~\ref{l:time_stepping_local} on each interval $I_j$, $j=1,\ldots,\cM-1$ to get 
\begin{align*}
&\bigg \|\sum_{j=1}^{\cM-1} \sum_{l=1}^\cN  \int_{t_{j,l-1}}^{t_{j,l}} W_h(r) \pi_h (f(\tt-r)-f(\tt-t_{j,l-1/2}))\, dr\bigg\|_{\dH^{2s}}\\
& \qquad \leq C \sum_{j=1}^{\cM-1} \tau_j^{5/2} \cN^{1/2} t_{j}^{\gamma-1} \| g_{tt} \|_{L^2(t_{j-1},t_j;\dH^{2s})} 
+ C \sum_{j=1}^{\cM-1} \tau_j^3 \cN t_{j}^{\gamma-2} \| g_t \|_{L^\infty(t_{j-1},t_j;\dH^{2s})},
\end{align*}
where we use the fact that $C^{-1} t_j \leq t_{j,l} \leq C t_j$ for some constant $C$ independent of $\cN$ and $\cM$.
Hence, a Cauchy-Schwarz inequality and the definitions of $t_j$ and $\tau_j$ yield 
\begin{align*}
&\bigg\|\sum_{j=1}^{\cM-1} \sum_{l=1}^\cN  \int_{t_{j,l-1}}^{t_{j,l}} W_h(r) \pi_h (f(\tt-r)-f(\tt-t_{j,l-1/2}))dr\bigg\|_{\dH^{2s}}\\
& \qquad \leq C\tt^{\frac32+\gamma} \cN^{-2}  \| g_{tt} \|_{L^2(0,T;\dH^{2s})}  \left( \sum_{j=1}^{\cM} 2^{-(3+2\gamma)(\cM-j)}\right)^{1/2}
\\
& \qquad + C \tt^{1+\gamma} \cN^{-2} \| g_t \|_{L^\infty(t_{j-1},t_j;\dH^{2s})} \sum_{j=1}^{\cM} 2^{-(1+\gamma)(\cM-j)}\\
& \qquad \leq C \cN^{-2} (\tt^{\frac32+\gamma} \| g_{tt} \|_{L^2(0,\tt;\dH^{2s})} + \tt^{1+\gamma} \| g_t \|_{L^\infty(0,\tt;\dH^{2s})}).
\end{align*}
This, together with the estimate for the first interval, implies
\begin{equation*}
\begin{split}
\| u_h(\tt)-u_h^{\cN}(\tt) \|_{\dH^{2s}} &\leq C\cN^{-2} \big( \tt^\gamma\| f\|_{L^\infty(0,\tt;\dH^{2s})}\\
& \qquad+ \tt^{3/2+\gamma} \| g_{tt} \|_{L^2(0,\tt;\dH^{2s})} + \tt^{1+\gamma} \| g_t \|_{L^\infty(0,\tt;\dH^{2s})}\big).
\end{split}
\end{equation*}
To conclude, we  observe that 
$$\| g_{t}
\|_{L^\infty(0,\tt;\dH^{2s})} = \|f_{t} \|_{L^\infty(0,\tt;\dH^{2s})},\qquad
 \| g_{tt} \|_{L^2(0,\tt;\dH^{2s})}= \| f_{tt}
\|_{L^2(0,\tt;\dH^{2s})}$$
 and that the embedding  $H^1(0,\tt) \subset
L^\infty(0,\tt)$ is continuous with norm independent of $\tt \geq \tt_0$.
\end{proof}

\subsection{A Sinc Approximation of the Contour Integral}
In view of \eqref{e:rel_num_nh}, one remaining problem is to compute 
$$
H_h(t,\tau):=L_h^{-\beta} \left( e_{\gamma,1}(-t^\gamma L_h^\beta) - e_{\gamma,1}(-(t+\tau)^\gamma L_h^\beta) \right)g_h
$$
for $t>0$, $\tau>0$ and $g_h \in \vh$. 
We proceed as in the homogeneous case discussed in Section~\ref{sinc}. 

 Let $N$ be a positive integer and let $k>0$ be a quadrature spacing.
For $t,\dt>0$ and $g_h\in \vh$, we
propose the following sinc approximation of $H_h(t,\tau)$:
\beq\label{e:sincapp2}
\bal
	Q_{h,k}^{N}(t,\tau)g_h := \frac{k}{2\pi i} \sum_{j=-N}^N & [e_{\tpow,1}({-t^\tpow z(y_j)^\spow})-e_{\tpow,1}({-(t+\dt)^\tpow z(y_j)^\spow})]\\
	&z(y_j)^{-\spow} z'(y_j)[(z(y_j)I-L_h)^{-1}g_h],
\eal
\eeq
where $z(y)$ for $y\in \RR$ is the hyperbolic contour \eqref{hc}.
With this, the computable approximation of the solution to the non-homogeneous problem becomes
\beq\label{e:fully}
	 u_{h,k}^{\cN,N}(\tt):= \sum_{j=1}^{\cM-1} \sum_{l=1}^{\cN} 
Q_{h,k}^N({t}_{j,l-1},\dt_j)\pih f(t-{t}_{j,l-\frac12}) .
\eeq
We start with the approximation  of $H_h(t,\tau)$ by $Q_{h,k}^N(t,\tau)$.
\begin{lemma}\label{t:sincquad}
Let  $t,\dt>0$, $s\in [0,1/2]$ and $d\in (0,\pi/4)$.
There exists a constant $C$ only depending on $d$, $b$, $\beta$, $\lambda_1$ such that for any $g_h \in \vh$,
$$
	\|(H_h(t,\dt)- Q_{h,k}^N(t,\dt))g_h\|_{\dH^{2s}}\le C t^{-1}\dt\left(e^{-\pi d/k}+e^{-\spow Nk}\right)\|g_h\|_{\dH_h^{2s}} .
$$
\end{lemma}
\begin{proof}
For $y\in B_d$, define 
$$h_\lambda(y,t,\dt)=z(y)^{-\spow}[e_{\tpow,1}(-t^\tpow z(y)^\spow)-e_{\tpow,1}(-(t+\dt)^\tpow z(y)^\spow)]z'(y)(z(y)-\lambda)^{-1}$$ 
and note that
$$
\bal
	&|e_{\tpow,1}(-t^\tpow z(y)^\spow)-e_{\tpow,1}(-(t+\dt)^\tpow z(y)^\spow)|\\
	& \qquad\le \int_t^{t+\dt} |z(y)^\beta s^{\tpow-1}e_{\tpow,\tpow}(-s^\tpow z(y)^\beta)|\, ds\le  Ct^{-1}\dt.
\eal
$$
Here we applied \eqref{ml-bound-scalar} replacing $z$ with $-z(y)^\spow s^\tpow$ 
so that $$|z(y)^\spow s^\tpow e_{\tpow,\tpow}(-s^\tpow z(y)^\spow)|\le C .$$
Hence, the desired estimate follows upon proceeding as in the proofs of Lemmas~\ref{l:contour-esti-i} and~\ref{l:sincquad}.
\end{proof}

%\red{We may not include the following remark since the first interval $[0,t_1]$ 
%is not involved in our time discretization scheme.
%\begin{remark}[First Step]\label{r:sincquad3} 
%Similar arguments guarantee that the approximation of $\mathcal H_{h}(t):=L_h^{-\spow} e_{\tpow,1}(-t^\tpow L_h^\spow) g_h$ by its corresponding sinc quadrature approximation
%$$
%\mathcal Q_{h,k}^N(t) := \frac{k}{2\pi i} \sum_{j=-N}^N e_{\tpow,1}({-t^\tpow z(y_j)^\spow}) z(y_j)^{-\spow} z'(y_j)[(z(y_j)I-L_h)^{-1}g_h];
%$$
%satisfies the error estimate
%$$\|(\mathcal H_h(t)-\mathcal Q_{h,k}^N(t))g_h\|_{\dH^{2s}}\le C\left( e^{-\pi d/k}+e^{-\spow Nk} \right)\|g_h\|_{\dH^{2s}} .$$
%where the constant $C$ is independent of $h$, $t$, $k$ and $N$.
%\end{remark}}

We are now in a position to prove the error estimate for the sinc quadrature on the non-homogeneous problem.
\begin{lemma}\label{t:sincquad2}
Let $\tt>0$, $s\in [0,1/2]$ and assume that $f \in L^\infty(0,\tt;\dH^{2s})$. 
Let $N$ be a positive integer, $d\in (0,\pi/4)$ and set $k=\sqrt{\frac{\pi d}{\spow N}}$.
Let $u_h^{\cN,\cM}$ be as in \eqref{e:space_time_solution} and let $u_{h,k}^{\cN,N}$ be as in \eqref{e:fully}.
There exists
a constant $C$ independent of $h$, $\tt$, $k$, $N$, $\cN$, $\cM$ satisfying
$$
	\|u_{h}^{\cN,\cM}(\tt) - u_{h,k}^{\cN,\cM}(\tt)\|_{\dH^{2s}} \le C\cM  e^{-\sqrt{\pi \beta d N}} \|f\|_{L^\infty(0,\tt;\dH^{2s})}.
$$
\end{lemma}
\begin{proof}
Note that both $u^{\cN,\cM}_h$ and $u^{\cN,\cM}_{h,k}$ are approximations starting at $t_1$ (the first interval $I_0=[0,t_1]$ is skipped). 
Hence, applying Lemma~\ref{t:sincquad} on each interval $I_{j,l}$ (i.e. with $\tau=\tau_j$, $t=t_{j,l}$ and $g_h = \pi_h f(\tt-t_{j,l-\frac12})$) for $j=1,...,\cM-1$ and $l=0,...,\cN$,  yields
$$
\bal
& \left\|\sum_{j=1}^{\cM-1} \sum_{l=1}^{\cN} (H_h({t}_{j,l-1},\dt_j)- Q_{h,k}^N({t}_{j,l-1},\dt_j))g_h({t}_{j,l-\frac12}) \right\|_{\dH^{2s}}\\
&\qquad \le C \cN \left|\sum_{j=1}^{\cM-1} \tau_j  t_{j}^{-1}\right|e^{-\sqrt{\pi \beta d N}}\|f\|_{L^\infty(0,\tt;\dH^{2s})}\\
&\qquad \le C\cM e^{-\sqrt{\pi \beta d N}}\|f\|_{L^\infty(0,\tt;\dH^{2s})},
\eal
$$
where we have used the definition \eqref{e:tjl} of $t_{j,l}$ to guarantee that $C^{-1} t_j \leq t_{j,l}\leq C t_j$ as well as the definition of $\tau_j = 2^{-(\cM-j)}\tt/\cN$. 
This is the desired result.
\end{proof}
%\begin{remark}
%We balance the two exponential above by letting $\pi d/k=\spow Nk$. Thus, for a fixed positive integer $N$, 
%we get $k=\sqrt{\frac{\pi d}{\spow N}}$ and the above error estimate becomes
%$$
%	\|U_h(t)-\mathcal U_h(t)\|\le C\max(1,t^{\tpow})e^{-\sqrt{\pi d\spow N}}\|f\|_{L^\infty(0,t;L^2)}.
%$$
%
%\end{remark}

\subsection{Total Error}
We summarize this section by the following total error estimate for the fully discrete approximation \eqref{e:fully} to the solution of the non-homogeneous problem.
Since $k=k(N)$ and $\cM=\cM(\cN)$, we denote by $u_{h}^{\cN,N}(\tt)$ the fully discrete solution \eqref{e:fully}.

\begin{theorem}[Total Error]\label{t:nfull}
Assume that Assumption~\ref{regularity} holds for  $\alpha\in (0,1]$.
Furthermore, let $\gamma \in (0,1)$, $\delta\geq 0$, $s \in [0,\min(1/2,\delta)]$  and let $\alpha^*$ be as in \eqref{e:astar}.
Let $\tt>0$,  $\cN$ a positive  integer and  $\cM=\left\lceil\frac{2\log_2{\cN}}{\tpow}\right\rceil$.
Let $N$ be a positive integer, $d\in (0,\pi/4)$ and set $k=\sqrt{\frac{\pi d}{\spow N}}$.
There exists a constant C independent of $h$, $\cN$, $\tt$ and $N$ such that for every $f\in H^2(0,T;\HH^{2\delta})$ we have
$$
\bal
\|u(\tt)- u_h^{\cN,N}(\tt)\|_{\HH^{2s}}&\le \tilde D(\tt)h^{2\alpha_*}\|f\|_{L^\infty(0,\tt;\HH^{2\delta})}+C\max(\tt^\gamma,\tt^{\frac32+\gamma}) \cN^{-2}  \|f\|_{H^2(0,\tt;\HH^{2s})}\\
&\qquad  +C\log_2(\cN) e^{-\sqrt{\pi d\spow N}}\|f\|_{L^\infty(0,\tt;\HH^{2s})} ,
\eal
$$
where $\tilde D(T)$ is given by \eqref{e:tildeD}.
\end{theorem}
\begin{proof}
This is in essence Lemmas~\ref{l:semirhs},~\ref{t:geo} and~\ref{t:sincquad2} together with the equivalence property between the  dotted spaces and interpolation spaces \eqref{e:interpolation_spaces} (see Proposition~\ref{p:equiv}).

\end{proof}

\begin{remark}[Choice of $\cN$ and $N$]
In practice, we  balance the three error terms in Theorem~\ref{t:nfull} by setting
$$
	N = c_1 (2\alal\ln(1/h))^2\quad\text{and}\quad \cN = c_2\lceil h^{-\alal}\rceil,
$$
for some positive constants $c_1$ and $c_2$ so that the total error behaves like $h^{2\alal}$.
We note that the number of the finite element systems that need to be
solved for the non-homogeneous problem is the same as 
for the homogeneous problem, i.e. $O(\ln(1/h)^2)$ complex systems (see the numerical illustration below).
\end{remark}

%%%%%%%%%%%%%%%%%%%%%%%%%%%%%%%%%%%%%%%%%%%%%%%%%%

\subsection{Numerical illustration}

To minimize the number of system solves in the 
 computation of \eqref{e:fully}, we rewrite
\begin{align*}
u_{h,k}^{\cN,N}(\tt) &= \sum_{j=1}^{\cM-1} \sum_{l=1}^{\cN} \frac{k}{2\pi i}\sum_{n=-N}^N 
\left( e_{\tpow,1}(-{t}_{j,l-1}^\tpow z(y_n)^\spow)-e_{\tpow,1}(-{t}_{j,l}^\tpow z(y_n)^\spow) \right)\\
&\qquad \qquad \qquad \qquad \qquad  z'(y_n) (z(y_n)I-L_h)^{-1} \pih f(T-{t}_{j,l-1/2})\\
&=\frac{k}{2\pi i}\sum_{n=-N}^N z(y_n)^{-\beta}z'(y_n)(z(y_n)I-L_h)^{-1} \cH_{n} ,
\end{align*}
where 
\beq
\cH_{n}:=\sum_{j=1}^{\cM-1} \sum_{j=1}^\cN
\left(e_{\gamma,1}(-{t}_{l,j-1}^\tpow
  z(y_n)^\spow)-e_{\tpow,1}(-{t}_{j,l}^\tpow z(y_n)^\spow) \right)\pih
f(T-{t}_{j,l- 1/2}).
\label{chsum}
\eeq
To implement the above we proceed as follows:
\begin{enumerate}[1)]
\item Compute the inner product vectors, i.e., the integral of $
  f(t-{t}_{j,l-1/2})$ against the finite element basis vectors, for all $(j,l)$.
\item For each, $n$:
\begin{enumerate}[a)] 
\item compute the sums in \eqref{chsum} but replacing $\pi_hf(T-t_{j,l-1/2})
  $ by the corresponding inner product vector, and 
\item compute $z(y_n)^{-\beta}z'(y_n)(z(y_n)I-L_h)^{-1} \cH_{n}$ by
  inversion of the corresponding stiffness matrix applied to the vector
  of Part~a).
\end{enumerate}
\item Sum up all contribution and multiply the result by $\frac{k}{2\pi i}$.
\end{enumerate}

We illustrate the error  behavior in time on a two dimensional problem
with
 domain  $\Omega=(0,1)^2$ and $L=-\Delta$ with homogeneous Dirichlet
 boundary conditions.
We set $\spow=0.5$ and consider the exact soltuion $u(t,x_1,x_2)=t^3\sin{(\pi x_1)}\sin{(\pi x_2)}$ which vanishes at $t=0$. 
This corresponds to
$$
f(x_1,x_2,t) = \left( \frac{\Gamma(4)}{\Gamma(4-\gamma)} t^{3-\gamma}  + t^3 (2\pi^2)^\beta \right)\sin{(\pi x_1)}\sin{(\pi x_2)}.
$$
We partition $\Omega$ using uniform triangles with the mesh size $h=2^{-5}\sqrt{2}$ and use $N= 400$ for the sinc quadrature parameter.
We also set $b=1$ in the hyperbolic contour \eqref{hc}.
In Figure~\ref{f:2Dnh} (left), we report $\|u(0.5)- Q_{h}^{\cN,N}(0.5)\|$ for $\cN = 2,4,8,16,32$
%$\cM=2,4,8,16,32,64$
 and different values of $\tpow$. In each cases, as predicted by Theorem~\ref{t:geo}, the rate of convergence $\cN^{-2}$ is observed.
For comparison, the approximation based on a uniform partition is also provided. In this case, the error decay  behaves like $\tau^{1+\gamma}$ (see Remark~\ref{r:uniform_step}).
\begin{figure}[hbt!]
 \begin{center}
    \begin{tabular}{cc}
\includegraphics[scale=.47]{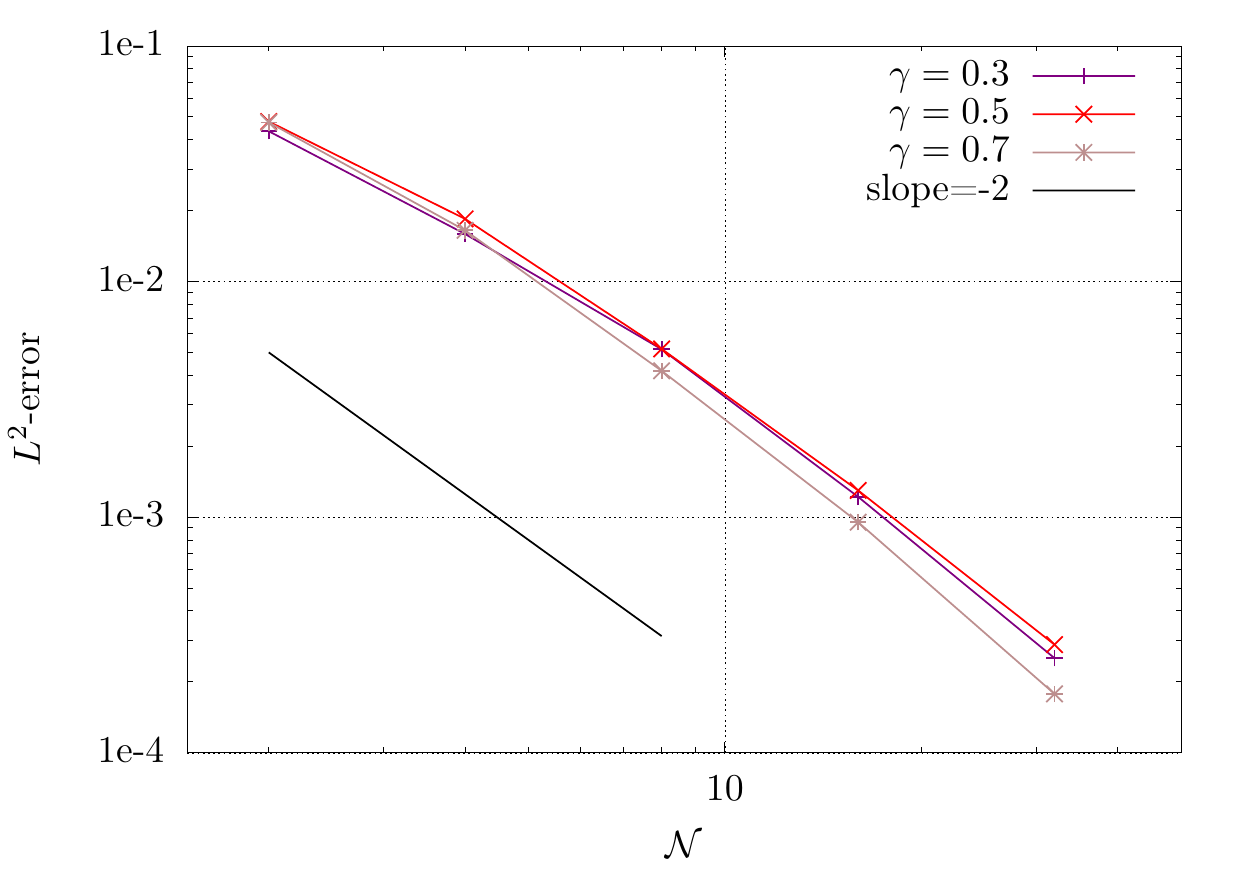} & \includegraphics[scale=.47]{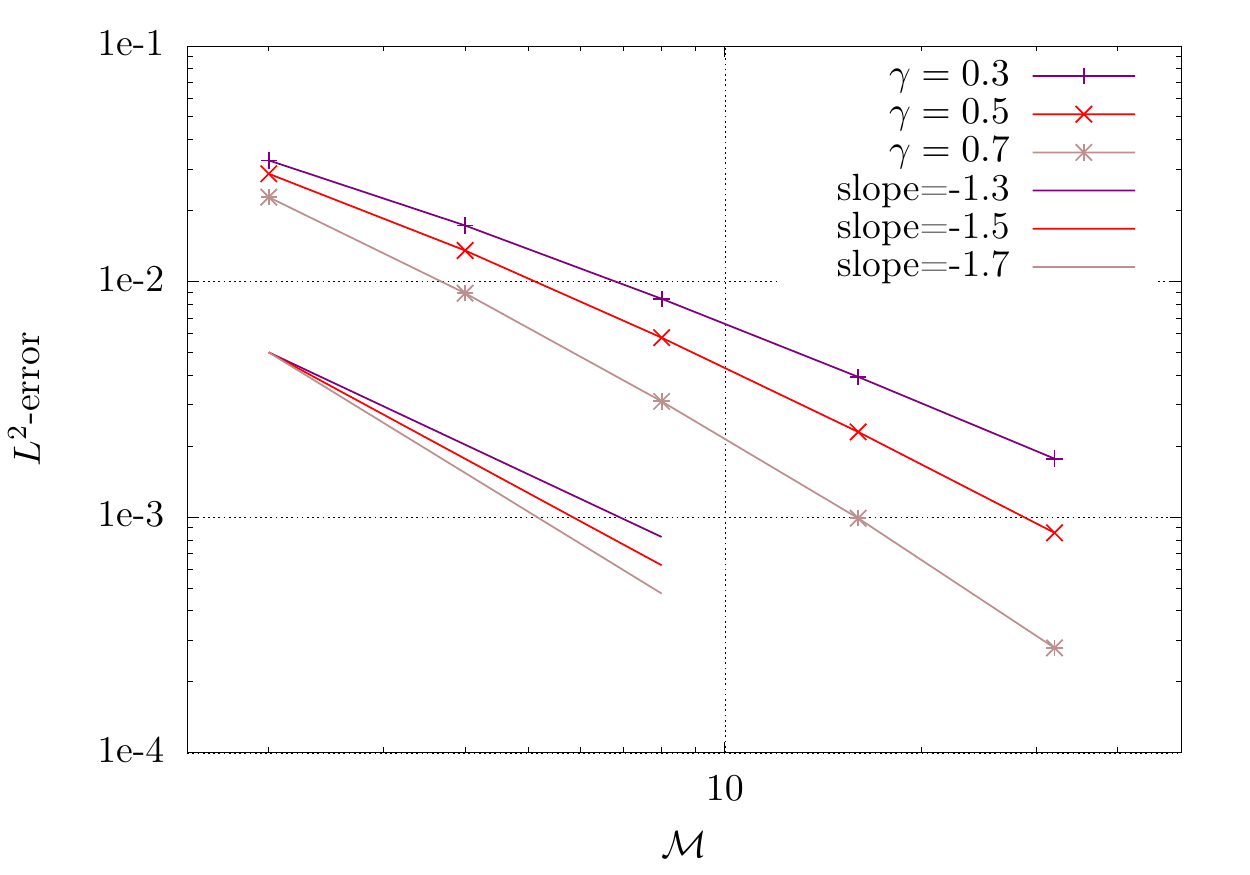} \\
    \end{tabular}
 \end{center}
    \caption{The left graph depicts for different values of $\gamma$, the $L^2$ error between $u(0.5)$ and the fully discrete approximation $u_{h}^{\cN,N}(0.5)$ as a function of $\cN$. The optimal rate of convergence $\cN^{-2}$ predicted by Theorem~\ref{t:geo} is observed.
    In contrast, when using uniform time stepping (right), the observed rate is $\tau^{1+\gamma}$ as announced in Remak~\ref{r:uniform_step}.
}
    \label{f:2Dnh}
\end{figure}

%%%%%%%%%%%%%%%%%%%%%%%%%%%%%%%%%%%%%%%%%%%%%%%%%%
\appendix
\section{Proof of Lemma~\ref{l:residue}}\label{a:lemma1}
The following lemma proved in \cite{BLP17} (see, Lemma 3.1 of \cite{BLP17})
and is instrumental in the proof of
Lemma~\ref{l:residue}.

\begin{lemma} There is a positive constant $C$ only depending on 
  $s\in [0,1]$
such that 
\beq|z|^{-s} \|T^{1-s} (z^{-1}I-T)^{-1}f\|\le C\|f\|,\Forall z\in
\cC,f\in L^2.
\label{lemeq}
\eeq
The same inequality holds $\vh$, i.e. with $T$ replaced by $T_h$ and $f \in \vh$.
\label{tz-bound}
\end{lemma}

\begin{proof}[Proof of Lemma~\ref{l:residue}]
%Rewriting the left hand side of \eqref{rz-bound} as
Noting that $$R_z(L)=(z I-L)^{-1}=T(zT-I)^{-1}$$ and 
 $$R_z(L_h)\pih=(z I-L_h)^{-1}\pih=(zT_h-I)^{-1}T_h\pih=(zT_h-I)^{-1}T_h ,$$ we obtain 
\begin{align*}
\pi_h R_z(L)-R_z(L_h)\pi_h %&=\pi_h(R_{z}(L)-R_{z}(L_h)\pi_h)\\
%&=\pi_h((z-L)^{-1}-(z-L_h)^{-1}\pih)\\
&=\pi_h(T(z T-I)^{-1}-(z T_h-I)^{-1}T_h)\\
%&=\pi_h(zT_h-I)^{-1}((zT_h-I)T-T_h(zT-I))(zT-I)^{-1}\\
&=\pi_h(zT_h-I)^{-1}(T_h-T)(zT-I)^{-1}\\
&=-z^{-2} ( T_h -z^{-1})^{-1}\pi_h(T-T_h)(T-z^{-1})^{-1},
\end{align*}
where for the last step we used the definition of $T_h$ to deduce that $\pi_h(zT_h-I)^{-1}=(zT_h-I)^{-1}\pi_h$.
We have left to prove:
\beq 
\|W(z)\|_{\dH^{2\delta}\rightarrow\dH^{2s}} \le C h^{2\tilde \alpha},
\label{w-bound}
\eeq
for a  constant $C$ is independent of $h$ and $z$ and where 
$$
W(z):= |z|^{-1-\tilde \alpha-s+\delta} (z T_h -I)^{-1}\pi_h(T-T_h)(z T-I)^{-1}.
$$

To show this, we write
\beq\bal
&\|W(z)\|_{\dH^{2\delta}\rightarrow\dH^{2s}}\\
\qquad & \le\underbrace{|z|^{-(1+\tpow)/2-s}\|( T_h
-z^{-1})^{-1}\pi_h\|_{\dot{H}^{1-\tpow}\rightarrow
  \dH^{2s}}}_{+;\mathrm{I}}\underbrace{\|(T-T_h)\|_{\dot{H}^{\alpha-1}\rightarrow\dot{H}^{1-\tpow}}}_{=:\mathrm{II}}\\
\qquad & \qquad \underbrace{|z|^{-(1+\alpha)/2+\delta}\|(
T-z^{-1})^{-1}\|_{\dot{H}^{2\delta}\rightarrow\dot{H}^{\alpha-1}}}_{=:\mathrm{III}},
\eal
\label{threet}
\eeq
where $\gamma :=2\alpha^*-\alpha$.
We estimate the three terms on the right hand side above separately.

We start with $\mathrm{III}$ and use the definition of the dotted spaces (see Section~\ref{ss:dotted}) to write
$$\bal
\|(T-z^{-1})^{-1}\|_{\dot{H}^{2\delta}\rightarrow\dot{H}^{\alpha-1}}
&=\sup_{w\in \dot{H}^{2\delta}} \frac {\|T^{(1-\alpha)/2} (T-z^{-1})^{-1} w\|}
{\|L^\delta w\|}\\
& = \sup_{\theta\in L^2} \frac {\|T^{(1-\alpha)/2} (T-z^{-1})^{-1}
  T^\delta \theta\|}
{\|\theta\|}\\ & =\|T^{1-[(1+\alpha)/2-\delta]} (T-z^{-1})^{-1}\|.
\eal
$$
Applying Lemma~\ref{tz-bound}  (recall that $\delta \in [0,(1+\alpha)/2]$ and $\alpha \in [0,1]$ so that $(1+\alpha)/2-\delta \in [0,1]$), 
we obtain
\beq
 \mathrm{III} = |z|^{-(1+\alpha)/2+\delta}\|(
T-z^{-1})^{-1}\|_{\dot{H}^{2\delta}\rightarrow\dot{H}^{\alpha-1}}\le 
C,
\label{thirdt}
\eeq
where $C$ is the constant in \eqref{lemeq}.

To estimate $\mathrm{I}$, we start with  the equivalence of norms \eqref{ineq:H_h_H} so that
$$
 \|( T_h-z^{-1})^{-1}\pi_h \|_{\dot{H}^{1-\tpow}\rightarrow \dH^{2s}}
%&\le C \|( T_h-z^{-1})^{-1}\pi_h \|_{\dot{H}^{1-\tpow}\rightarrow \dH^{2s}_h}\\
\le C\|( T_h-z^{-1})^{-1}\|_{\dot{H}_h^{1-\tpow}\rightarrow \dH^{2s}_h}
\|\pi_h\|_{\dot{H}^{1-\tpow}\rightarrow \dot{H}_h^{1-\tpow}}.
$$
Whence, the stability of the $L^2$ projection \eqref{pih-bound} together with the equivalence property between dotted spaces and interpolation spaces (Proposition~\ref{p:equiv})  as well as the definition of the discrete dotted space norm \eqref{e:dotted_discrete_norm} lead to
\begin{equation}
 \|( T_h-z^{-1})^{-1}\pi_h \|_{\dot{H}^{1-\tpow}\rightarrow \dH^{2s}}
%&\le C \|( T_h-z^{-1})^{-1}\pi_h \|_{\dot{H}^{1-\tpow}\rightarrow \dH^{2s}_h}\\
%&\leq C \|( T_h-z^{-1})^{-1}\|_{\dot{H}_h^{1-\tpow}\rightarrow \dH^{2s}_h}\\
\leq C \|T_h^{1-[(1+\tpow)/2+s]}( T_h-z^{-1})^{-1}\|.
\label{firstt}
\end{equation}
We recall that $\alpha \in (0,1]$ and $\tpow=2\tilde \alpha-\alpha$ so that $(1+\tpow)/2+s\in (0,1]$.
Hence, Lemma~\ref{tz-bound} ensures the following estimate:
\beq
\mathrm{I} = |z|^{-(1+\tpow)/2-s}\|
(T_h-z^{-1})^{-1}\pi_h \|_{\dot{H}^{1-\tpow}\rightarrow \dH^{2s}}
\le C.
\label{firstfinal}
\eeq

For the remaining term,  Proposition~\ref{prop:T_error} with $2s = 1-\tpow$ gives
$$
\text{II} \leq C h^{\alpha + \min(\alpha,\gamma)} = C h^{\min(2\alpha,2\tilde \alpha)} =  C h^{2\tilde \alpha}.
$$
Combining the above estimate with \eqref{thirdt} and \eqref{firstfinal} yields \eqref{w-bound} and completes the proof.
\end{proof}

\section{Sinc quadrature Lemma.}\label{a:lemma2}

The results of the next lemma are contained in the proof of Theorem~4.1
of \cite{BLP17}.

\begin{lemma}\label{l:contour-esti-ii}
Let $0<d<\pi/4$ and $\lambda>\lambda_1$. 
let $z(y)$ be defined by \eqref{hc} and $B_d =\left\{z \in \mathbb C : \ \Im(z)< d\right\}.$
The following assertions hold.
\begin{enumerate}[(a)]
\item There exists a constant $C>0$ only depending on $\lambda_1$, $b$ and $d$ such that
\beq
|z(y)-\lambda|\ge C  \Forall y\in \bar B_{d};
\label{glambdai}
\eeq
\item There exists a constant $C>0$ only depending on $\lambda_1$, $b$ and $d$ such that 
$$
|z'(y)(z(y)-\lambda)^{-1}|\le C \Forall y\in B_d;
$$
\item There is a constant $C>0$ only depending on $b$, $d$ and $\beta$ such that
$$
\Re(z(y)^\spow)\ge C 2^{-\beta} e^{\spow|\Re y|} \Forall y\in B_d.
$$
\end{enumerate}
\end{lemma}

\begin{proof}[Proof of the Lemma~\ref{l:contour-esti-i}]
From the expression  (4.13) of $\Re(z(y))$ in \cite{BLP17}, we deduce that $\Re(z(y))$ is strictly positive for
$y\in \bar B_{d}=\{w\in \mathbb C :\ \Im(w)\le d\}$.  It follows from
this and 
Part~(a) of Lemma~\ref{l:contour-esti-ii} that Condition~(i) of
Definition~\ref{class_SB} holds for $g_\lambda(\cdot,t)$ for $\lambda
\ge \lambda_1$ and $t>0$.

We now give a proof of (ii) and (iii) of Definition~\ref{class_SB} simultaneously.
Note that Part (b) in Lemma~\ref{l:contour-esti-ii} together with \eqref{ml-bound-scalar} imply that for $y\in \overline B_d$,
$$|g_\lambda(y,t)|\le \frac{C}{1+t^\tpow |z(y)^\spow|} \le\frac{C}{1+t^\tpow |\Re (z(y)^\spow)|} .$$
Furthermore, the estimate on $\Re(z(y)^\beta)$ in Part (c) of
Lemma~\ref{l:contour-esti-ii} 
yields
\begin{equation}\label{e:estim_app}
	|g_\lambda(y,t)|\le \frac{C}{1+t^\tpow \kappa2^{-\beta} e^{\spow|\Re y|}}	\le C(\spow,d,b)t^{-\tpow}e^{-\spow|\Re y|} .
\end{equation}
This guarantees that 
$$
	\int_{-d}^{d} |g_\lambda(u+iw,t)|\, dw\leq C(\spow,d,b) t^{-\tpow}
$$
and
$$
\bal
	N(B_d)&=\int_{-\infty}^\infty \left(|g_\lambda(u+id)|+|g_\lambda (u-id)| \right) du\\
	&\le t^{-\tpow}C(\spow,d,b)\int_0^\infty e^{-\spow y}\, dy
	\le C(\spow,d,b)t^{-\tpow}
\eal
$$
which yield (ii), (iii) and the bound on $N(B_d)$.
\end{proof}

\bibliographystyle{plain}

\end{document}